\documentclass[12pt,leqno]{article}
\usepackage{amssymb}
\usepackage[mathscr]{eucal}
\usepackage{amsmath,amssymb,latexsym,theorem,enumerate,ifthen}
\usepackage[english]{babel}
\usepackage{color,url}
\usepackage{appendix}
\usepackage{graphicx}
\usepackage{subcaption}

\newcommand{\NN}{\mathbb{N}}

\newcommand{\RR}{\mathbb{R}}

\newcommand{\ZZ}{\mathbb{Z}}

\newcommand{\bone}{{\boldsymbol{1}}}

\newcommand{\cA}{{\mathcal A}}

\newcommand{\cP}{{\mathcal P}}

\newcommand{\EE}{\operatorname{\mathbb{E}}}

\newcommand{\PP}{\operatorname{\mathbb{P}}}

\newcommand{\vare}{\varepsilon}

\renewcommand{\leq}{\leqslant}
\renewcommand{\geq}{\geqslant}

\newcommand{\proofend}{\hfill\mbox{$\Box$}}

\numberwithin{equation}{section}

\theoremstyle{change} \theorembodyfont{\em}
\newtheorem{Lem}{Lemma.}[section]
\newtheorem{Thm}[Lem]{Theorem.}
\newtheorem{Pro}[Lem]{Proposition.}
\newtheorem{Cor}[Lem]{Corollary.}
\newtheorem{Def}[Lem]{Definition.}

\theorembodyfont{\rm}
\newtheorem{Rem}[Lem]{Remark.}
\newtheorem{Ex}[Lem]{Example.}

\def\OnlyOnArXiv#1#2{\ifthenelse{\equal{#1}{Y}}{#2}{}}

\newenvironment{proof}{\noindent{\bf Proof.}}{\proofend}

\begin{document}

\begin{center}
 {\bfseries\Large On H\"older continuity and $p^\mathrm{th}$-variation function\\[1mm]
 of Weierstrass-type functions}

\vspace*{3mm}

{\sc\large
  M\'aty\'as $\text{Barczy}^{*,\diamond}$,
  Peter $\text{Kern}^{**}$ }

\end{center}

\vskip0.2cm

\noindent
 * HUN-REN–SZTE Analysis and Applications Research Group,
   Bolyai Institute, University of Szeged,
   Aradi v\'ertan\'uk tere 1, H--6720 Szeged, Hungary.

\noindent
 ** Faculty of Mathematics and Natural Sciences, Heinrich Heine University D\"usseldorf,
    Universit\"atsstra{\ss}e 1, D--40225 D\"usseldorf, Germany.

\noindent E-mails: barczy@math.u-szeged.hu (M. Barczy),
                   kern@hhu.de (P. Kern).

\noindent $\diamond$ Corresponding author.

\vskip0.2cm

\renewcommand{\thefootnote}{}
\vspace*{0.2cm}
\footnote{This research was carried out when M\'aty\'as Barczy visited Heinrich Heine University D\"usseldorf in June and July, 2024
 thanks to a scholarship in the programme Research Stays for University Academics and Scientists, 2024,
 granted by the German Academic Exchange Service (DAAD).}

\vspace*{-10mm}

\begin{abstract}
We study H\"older continuity, $p^\mathrm{th}$-variation function and Riesz variation of Weierstrass-type functions
 along the sequence of $b$-adic partitions, where $b>1$ is an integer.
By a Weierstrass-type function, we mean that in the definition of the well-known Weierstrass function,
 the power function is replaced by a submultiplicative function, and the Lipschitz continuous
 cosine and sine functions are replaced by a general periodic H\"older continuous function.
\end{abstract}


\section{Introduction}
\label{section_intro}

Investigation of properties such as continuity, Lipschitz continuity, H\"older continuity, differentiability and bounded variation
 for real functions has a long tradition in classical analysis.
Studying these properties for sample paths of stochastic processes has also attracted the attention of many researchers in stochastic analysis.
Recently, Gatheral et al.\ \cite{GatJaiRos} have pointed out the fact that the empirical daily realized variance values of some stocks and stock price indices
 are much more likely to be sampled from a stochastic process having rough sample paths rather than smooth ones.
To measure the degree of roughness of a continuous function $g:[0,1]\to\RR$, Gatheral et al.\ \cite[Section 2.1]{GatJaiRos}
 investigate the quantity
 \begin{align}\label{help1}
   \sum_{i=0}^{n-1} \vert g(t_{i+1}) - g(t_i)\vert^p,
 \end{align}
 where $n$ is a positive integer, $0=t_0<t_1<\cdots<t_{n-1}<t_n=1$ is a partition of $[0,1]$ and $p\geq 1$ is a parameter.
As it is explained in the introduction of Schied and Zhang \cite{SchZha1},
 the intuition is that, we may expect the existence of a number $q\in[1,\infty)$ such that the sums in \eqref{help1} converge to zero for $p>q$
  and diverge for $p<q$ (in the latter case, provided that $q>1$) as the mesh of the partition in question tends to zero.
In case of the sequence of $b$-adic partitions (where $b>1$ is an integer), we will explain this intuition in mathematical terms
 in Lemma \ref{Lem_change_point}.
For the sample paths of a Wiener process, in case of the dyadic partitions,
 the corresponding value of $q$ is equal to $2$ almost surely
 (following, e.g., from Rogers \cite[Section 2]{Rog}),
 whereas values larger than 6 are reported in Gatheral et al.\
 \cite{GatJaiRos} for the empirical daily realized variance values mentioned before.

In this paper, we investigate H\"older continuity, $p^\mathrm{th}$-variation function
 (see Definition \ref{Def_cont_pth_variation}) and Riesz variation (see Definition \ref{Def_Riesz_var})
 of Weierstrass-type functions defined in \eqref{help7} along the sequence of $b$-adic partitions, where $b>1$ is an integer.
By a Weierstrass-type function, we mean that in the definition of the well-known Weierstrass function,
 the power function is replaced by a submultiplicative function, and the Lipschitz continuous
 cosine and sine functions are replaced by a general periodic H\"older continuous function.
Our results extend some of the recent results of Schied and Zhang \cite{SchZha1,SchZha2}.

Throughout this paper, let $\NN$, $\ZZ_+$, $\RR$, $\RR_+$ and $\RR_{++}$ denote the sets of positive integers,
 non-negative integers, real numbers, non-negative real numbers and positive real numbers, respectively.
An interval $I\subseteq \RR$ is called non-degenerate if it contains at least two distinct points.
For real numbers $x,y\in\RR$, let $x\wedge y:=\min(x,y)$.
All the random variables will be defined on a common probability space $(\Omega,\cA,\PP)$.
Given a non-degenerate interval $I$ of $\RR$, a function $g:I\to\RR$ is called H\"older continuous with exponent $\mu\in(0,1]$ 
 if there exists $C>0$ such that
 \[
           \vert g(x) - g(y)\vert \leq C \vert x-y\vert^\mu, \qquad x,y\in I.
 \]
In case of $\mu=1$, we say that $g$ is Lipschitz continuous.
We mention that we excluded the case $\mu>1$ in the definition of H\"older continuity, since then $g$ is a constant function.

Following Cont and Perkowski \cite[Definition 1.1 and Lemma 1.3]{ConPer} (see also Schied and Zhang \cite{SchZha1, SchZha2}),
 we introduce the notion of continuous $p^\mathrm{th}$-variation function of a continuous real-valued function
 along the sequence of $b$-adic partitions (where $b\in\NN\setminus\{1\}$ and $p\geq 1$), see Definition \ref{Def_cont_pth_variation}.
Their investigation is motivated by the fact that F\"ollmer's pathwise It\^{o} calculus (see F\"ollmer \cite{Fol})
 may be extended to stochastic processes with irregular sample paths
 in a strictly pathwise setting using the concept of $p^{\mathrm{th}}$-variation function along the sequence of $b$-adic partitions.
In particular, their results also apply to the sample paths of a fractional Wiener process with arbitrary Hurst exponent.
We also mention that, very recently, Das et al.\ \cite[Theorem 3.3]{DasKimLim} have developed a construction 
 of a continuous function with a prescribed, possibly nonlinear continuous $p^\mathrm{th}$-variation function (where $p>1$) 
 starting from a continuous function having a linear $p^\mathrm{th}$-variation function.
Das et al.\ \cite[Theorem 3.7]{DasKimLim} have also shown that the set of H\"older continuous functions
 with a given exponent in $(0,1)$ and having a prescribed continuous $p^\mathrm{th}$-variation function (where $p>1$) 
 is dense in the set of continuous functions defined on $[0,1]$, provided that such a H\"older continuous function exists.

\begin{Def}\label{Def_cont_pth_variation}
Let $g:[0,1]\to\RR$ be a continuous function, $b\in\NN\setminus \{1\}$, $p\geq 1$,
 and $\Pi_n:=\{kb^{-n}: k=0,1,\ldots,b^n\}$, $n\in\NN$, be the (refining) sequence of $b$-adic partitions of $[0,1]$.
If there exists a continuous function $\langle g \rangle^{(p)}:[0,1]\to \RR_+$ such that
 \begin{align}\label{pth_var_Def}
  V^{p,t}_n(g) := \sum_{k=0}^{\lfloor tb^n\rfloor}\vert g((k+1)b^{-n}) - g(kb^{-n})\vert^p
      \to \langle g \rangle^{(p)}(t) \qquad \text{as \ $n\to\infty$}
 \end{align}
 for all $t\in[0,1]$, then the function $\langle g \rangle^{(p)}$ is said to be the continuous $p^\mathrm{th}$-variation function
 of $g$ along the sequence of partitions $\Pi_n$, $n\in\NN$.
\end{Def}

Concerning the notations in Definition \ref{Def_cont_pth_variation},
we mention that, in the literature, $V^{p,t}_n(g)$ 
and $\langle g \rangle^{(p)}(t)$ are sometimes also denoted by $[g]_n^{(p)}(t)$
and $[g]^{(p)}(t)$, respectively.
For an extension of the notion of the continuous $p^\mathrm{th}$-variation function (given in Definition \ref{Def_cont_pth_variation})
 of a continuous real-valued function along some sequence of partitions to a continuous $\RR^d$-valued function (where $d\in\NN$) ,
 see Cont and Jin \cite[Defintion 4.1]{ConJin} and Cont and Perkwoski \cite[Defintion 4.1]{ConPer}.

The next two remarks are devoted to highlight the assumptions of Definition \ref{Def_cont_pth_variation}
 and to point out some easy consequences, among others, the fact that the convergence in \eqref{pth_var_Def} is uniform on $[0,1]$ as well.

\begin{Rem}\label{Rem2}
(i) In the sum $V^{p,t}_n(g)$ in \eqref{pth_var_Def}, the function $g$ defined on $[0,1]$ is formally evaluated at $1+b^{-n}>1$ if $t=1$ and $k=b^n$.
To handle this, we assume here and in the sequel that when we calculate $V_n^{p,t}(g)$ for a function
 $g$ defined on $[0,1]$ we extended the domain of $g$ to $\RR_+$ by setting
 $g(t):=g(1)$ for $t>1$.
It implies that if $t=1$ and $k=b^n$, then the summand in \eqref{pth_var_Def} corresponding to $k=\lfloor tb^n\rfloor$ is zero, 
 since $g((k+1)b^{-n}) - g(kb^{-n})=0$.
Further, note that 
 \[
    V^{p,0}_n(g) = \vert g(b^{-n}) - g(0)\vert^p\to 0 \qquad \text{as $n\to\infty$,}
 \]
 since $g$ is continuous, yielding that the limit in \eqref{pth_var_Def} for $t=0$ always exists and is equal to $0$.
 
(ii) For any $n\in\NN$ and $p\geq 1$, the function $[0,1]\ni t\mapsto V^{p,t}_n(g)$ is monotone increasing and càdlàg.
Hence, supposing that the limit in \eqref{pth_var_Def} exists for all $t\in[0,1]$, we have 
 $\langle g \rangle^{(p)}(t_1)\leq \langle g \rangle^{(p)}(t_2)$ for $0\leq t_1\leq t_2\leq 1$,
 i.e., $\langle g \rangle^{(p)}$ is a monotone increasing function on $[0,1]$.
However, it is not sure that $\langle g \rangle^{(p)}$ is continuous even if $g$ is continuous 
 (see Schied \cite[page 979]{Sch}).
This is the reason for supposing the continuity of  $\langle g \rangle^{(p)}$ in Definition \ref{Def_cont_pth_variation}.

(iii) Suppose that $g$ has a continuous $p^\mathrm{th}$-variation function
 along the sequence of partitions $\Pi_n$, $n\in\NN$, given in Definition \ref{Def_cont_pth_variation}.
In the case $\langle g \rangle^{(p)}(1)\ne 0$, for each $n\in\NN$ and $p\geq 1$, we have $F_n^{(p)}:\RR\to[0,1]$,
 \[
  F_n^{(p)}(t):= \begin{cases}
                  0 & \text{if $t<0$,}\\
                  \frac{V^{p,t}_n(g)}{\langle g \rangle^{(p)}(1)} & \text{if $t\in[0,1)$,}\\
                  1 & \text{if $t\geq 1$,}
                \end{cases}
 \]
 is a (right-continuous) distribution function, and, by taking the limit as $n\to\infty$,
 it converges pointwise to the continuous distribution function $F^{(p)}:\RR\to[0,1]$,
 \[
   F^{(p)}(t) :=  \begin{cases}
                  0 & \text{if $t<0$,}\\
                  \frac{\langle g \rangle^{(p)}(t)}{\langle g \rangle^{(p)}(1)} & \text{if $t\in[0,1)$,}\\
                  1 & \text{if $t\geq 1$.}
                \end{cases}
 \]
It is known that this implies that the convergence in \eqref{pth_var_Def} is uniform on $[0,1]$ as well.
For a condensed version of this argument, see the paragraph after Lemma 1.3 in Cont and Perkowski \cite{ConPer}.
In the case $\langle g \rangle^{(p)}(1)=0$, using part (ii), we have that 
 $0\leq \langle g \rangle^{(p)}(t)\leq \langle g \rangle^{(p)}(1)=0$ for all $t\in[0,1]$, yielding that  
 $\langle g \rangle^{(p)}(t)=0$, $t\in[0,1]$.
Consequently, to prove that the convergence \eqref{pth_var_Def} is uniform on $[0,1]$ in the case $\langle g \rangle^{(p)}(1)=0$ as well, 
 it is enough to check that for any sequence $(t_n)_{n\in\NN}$ in $[0,1]$ converging to $t\in[0,1]$, we have that 
 $V^{p,t_n}_n(g)\to 0$ as $n\to\infty$.
Using again part (ii), we get that $0\leq V^{p,t_n}_n(g) \leq  V^{p,1}_n(g)$, $n\in\NN$, 
 where $V^{p,1}_n(g)\to\langle g \rangle^{(p)}(1)=0$ as $n\to\infty$.
By the sandwich theorem, it implies that $V^{p,t_n}_n(g)\to 0$ as $n\to\infty$, as desired.

(iv) 
We draw the attention to the fact that the function $\langle g \rangle^{(p)}$ in Definition \ref{Def_cont_pth_variation}
 may depend not only on the function $g$ but also on the underlying sequence of refining partitions
 (in our case, on the parameter $b$), but we do not denote this dependence.
For example, Schied \cite[Proposition 2.7]{Sch} constructed a real-valued continuous function
 defined on $[0,1]$ of which the continuous $2^\mathrm{nd}$-variation (quadratic variation) functions
 along the sequences of (refining) partitions $\{k4^{-n}: k=0,1,\ldots,4^n\}$, $n\in\NN$,
 and $\{\frac{k}{2} 4^{-n}: k=0,1,\ldots,2\cdot4^n\}$, $n\in\NN$, respectively,
 exist and are different, but the continuous quadratic variation function along the sequence of dyadic
 partitions $\{k2^{-n}: k=0,1,\ldots,2^n\}$, $n\in\NN$, does not exist.
\proofend
\end{Rem}

\begin{Rem}
(i) For historical fidelity, we note that  Cont and Perkowski \cite[Definition 1.1 and Lemma 1.3]{ConPer} introduced  
  the notion of continuous $p^\mathrm{th}$-variation function along a given sequence of partitions
  for $p>0$, while in Definition \ref{Def_cont_pth_variation},
  we only consider the case $p\geq 1$.
Our heuristic reason for this restriction is that the case $p=1$ 
  (which is roughly speaking the case of bounded variation along a given sequence of partitions) 
  is a kind of ``nice'' property.  
If a function is not of bounded variation, then one may ask for a $p^\mathrm{th}$-variation function 
 with some $p>1$ when the increments (of the given function) less than $1$ are decreased by taking their $p^\mathrm{th}$ power. 
However, it makes no sense to consider some $p\in(0,1)$, since then taking $p^\mathrm{th}$ power of the increments in question
 would just increase them.

(ii)
The sequence $(V^{1,1}_n(g))_{n\in\NN}$ is monotone increasing, which can be checked using the triangle inequality and the fact that
 $\{kb^{-n} : k=0,1,\ldots,b^n\}$, $n\in\NN$, is the refining sequence of $b$-adic partitions of $[0,1]$.
Consequently, if $(V^{1,1}_n(g))_{n\in\NN}$ is bounded as well, then the limit $\lim_{n\to\infty} V^{1,1}_n(g)$ exists in $\RR_+$,
 i.e., the limit in \eqref{pth_var_Def} for $t=1$ exists.
Moreover, if the limit of $V^{p,t}_n(g)$ as $n\to\infty$ exists (where $t\in(0,1]$) along the sequence of
 $b$-adic partitions of $[0,1]$, then it also exists along the sequence of $b^N$-adic partitions of $[0,1]$
 for each $N\in\NN$, and the two limits coincide, since the later sequence is a subsequence of the original one (corresponding to $b$-adic partitions).

(iii)
If a continuous function $g:[0,1]\to\RR$ has a continuous $p^\mathrm{th}$-variation function
 along the sequence of partitions $\Pi_n$, $n\in\NN$, given in Definition \ref{Def_cont_pth_variation},
 then it does not necessarily have finite $p$-variation in the usual sense (also called in Wiener's sense,
  see Appell et al.\ \cite[Definition 1.31]{AppBanMer}).
For more details, see Cont and Perkowski \cite[Remark 1.2]{ConPer}.
\proofend
\end{Rem}

The content of the next lemma can be found in some papers, see, e.g.,
 the paragraph after Definition 4.4 on page 15 in Bayraktar et al.\ \cite{BayDasKim1}
 or Cont and Perkowski in \cite[part (2) of Remark 1.2]{ConPer}.
For completeness, we also provide a proof.
A similar statement for stochastic processes with continuous sample paths can be found, e.g., 
 in Karatzas and Shreve \cite[Chapter 1, Problem 5.11]{KarShr}.

\begin{Lem}\label{Lem_change_point}
Let $g:[0,1]\to\RR$ be a continuous function, $p\geq 1$, $b\in\NN\setminus\{1\}$,
 and suppose that $g$ has a continuous $p^\mathrm{th}$-variation function
 along the sequence of $b$-adic partitions $\Pi_n$, $n\in\NN$, given in Definition \ref{Def_cont_pth_variation}.
\begin{enumerate}
 \item[(i)] If $r>p$, then $g$ has a continuous $r^\mathrm{th}$-variation function
       along the sequence of partitions $\Pi_n$, $n\in\NN$, such that $\langle g \rangle^{(r)}(t)=0$, $t\in[0,1]$.
 \item[(ii)] If $1\leq r<p$ and $\langle g \rangle^{(p)}(t)>0$, $t\in(0,1]$,
       then
       \[
        \lim_{n\to\infty} V_n^{r,t}(g) = \lim_{n\to\infty}\sum_{k=0}^{\lfloor tb^n\rfloor} \vert g((k+1)b^{-n}) - g(kb^{-n})\vert^r = \infty, \qquad t\in(0,1].
       \]
\end{enumerate}
 \end{Lem}

\noindent{\bf Proof.}
(i) Suppose that $r>p$.
For all $t\in[0,1]$, we have
 \begin{align*}
  V_n^{r,t}(g)&=\sum_{k=0}^{\lfloor tb^n\rfloor} \vert g((k+1)b^{-n}) - g(kb^{-n})\vert^r\\
  &\leq \left(  \sup_{k\in\{0,1,\ldots,\lfloor tb^n\rfloor\}}  \vert g((k+1)b^{-n}) - g(kb^{-n})\vert  \right)^{r-p}
           \sum_{k=0}^{\lfloor tb^n\rfloor} \vert g((k+1)b^{-n}) - g(kb^{-n})\vert^p\\
  &\to 0\cdot \langle g \rangle^{(p)}(t) = 0  \qquad \text{as $n\to\infty$,}
 \end{align*}
 since $g$ is uniformly continuous on $[0,1]$.

(ii) Suppose that $1\leq r<p$ and $\langle g \rangle^{(p)}(t)>0$, $t\in(0,1]$.
In what follows, let $t\in(0,1]$ be arbitrarily fixed.
We assumed that
 \[
  \lim_{n\to\infty} \sum_{k=0}^{\lfloor tb^n\rfloor} \vert g((k+1)b^{-n}) - g(kb^{-n})\vert^p
     = \langle g \rangle^{(p)}(t)>0,
 \]
and, hence for all sufficiently large $n\in\NN$, there exists a $k_n\in\{0,1,\ldots,\lfloor tb^n\rfloor\}$
 such that $\vert g((k_n+1)b^{-n}) - g(k_nb^{-n})\vert>0$.
Therefore
 \[
  \sup_{k\in\{0,1,\ldots,\lfloor tb^n\rfloor\}}  \vert g((k+1)b^{-n}) - g(kb^{-n})\vert  > 0
 \]
 for all sufficiently large $n\in\NN$.
Consequently,  $1\leq r<p$ implies that
 \begin{align}\label{help46}
 \begin{split}
  V_n^{r,t}(g)&= \sum_{k=0}^{\lfloor tb^n\rfloor} \vert g((k+1)b^{-n}) - g(kb^{-n})\vert^r\\
  & \geq \left(  \sup_{k\in\{0,1,\ldots,\lfloor tb^n\rfloor\}}  \vert g((k+1)b^{-n}) - g(kb^{-n})\vert  \right)^{r-p}
           \sum_{k=0}^{\lfloor tb^n\rfloor} \vert g((k+1)b^{-n}) - g(kb^{-n})\vert^p
 \end{split}
  \end{align}
 for all sufficiently large $n\in\NN$.
The right hand side of \eqref{help46} tends to $\infty \cdot \langle g \rangle^{(p)}(t) = \infty$ as $n\to\infty$,
  since $g$ is uniformly continuous on $[0,1]$.
Hence the left hand side of \eqref{help46} also tends to $\infty$, yielding the assertion.
\proofend

\begin{Rem}\label{Rem4}
In the proof of Lemma \ref{Lem_change_point}, the assumption that $g$ is continuous is effectively used.
Further, in Lemma \ref{Lem_change_point}, it is assumed that the continuous function $g:[0,1]\to\RR$ has a
 continuous $p^\mathrm{th}$-variation function along the sequence of $b$-adic partitions $\Pi_n$, $n\in\NN$,
 given in Definition \ref{Def_cont_pth_variation}, where $p\geq 1$.
Note also that, for an arbitrary continuous function $g:[0,1]\to\RR$, it can happen that the limit in \eqref{pth_var_Def} does not exist
 for some $t\in(0,1]$.
To overcome this difficulty, given $p\geq 1$ and $\Pi_n$, $n\in\NN$, Das and Kim \cite[Definition 2.4]{DasKim}
 considered the set of those continuous functions $g:[0,1]\to\RR$ for which $\limsup_{n\to\infty} V_n^{p,1}(g)<\infty$,
 and showed that this space is a Banach space furnished with an appropriate norm, see  Das and Kim \cite[Proposition 2.5]{DasKim}.
Furthermore, given a continuous function $g:[0,1]\to\RR$,
 Das and Kim \cite[Definition 2.3]{DasKim} introduced the variation index of $g$
 along the sequence $\Pi_n$, $n\in\NN$, defined as
 \[
   p_g:=\inf\{ p\geq 1 : \limsup_{n\to\infty} V_n^{p,1}(g) < \infty\},
 \]
 for which it was checked that
 \[
   \limsup_{n\to\infty} V_n^{q,1}(g) = \begin{cases}
                                         0 & \text{if $q>p_g$,}\\
                                         \infty & \text{if $1\leq q < p_g$ (provided that $p_g>1$).}
                                       \end{cases}
 \]
Hence the above recalled result of Das and Kim \cite{DasKim} resembles to Lemma \ref{Lem_change_point} under weaker conditions.
We also mention that, very recently, Bayraktar et al.\ \cite[Lemma 4.3]{BayDasKim2} have derived an inequality 
 involving the variation index and the so-called critical H\"older exponent of a continuous real-valued function defined on $[0,1]$.
\proofend
\end{Rem}

In what follows, we recall some recent results on continuous $p^\mathrm{th}$-variation functions
 of signed Takagi-Landsberg functions and Weierstrass-type functions due to Mishura and Schied \cite{MisSch}
 and Schied and Zhang \cite{SchZha1, SchZha2}.
These types of functions are continuous, but, in general, nowhere differentiable.
For a good historical overview on continuous, but nowhere differentiable real-valued functions
defined on $\RR$, including Weierstrass-type, Takagi–type and Bolzano-type functions,
 see Chapter 1 in Jarnicki and Pflug \cite{JarPfl} and Kucharski \cite{Kuc}.
For the history of Takagi function and its generalizations, see also Allaart and Kawamura \cite{AllKaw}.
Here we only note that, according to \cite{JarPfl}, a continuous, but nowhere differentiable
 real-valued function defined on $\RR$ was first publicly accessible
 in 1872 due to Weierstrass.
Namely, it was shown that the function $\RR\ni x\mapsto \sum_{m=1}^\infty a^m \cos(b^m\pi x)$,
 where $a\in(0,1)$ and $b$ is an odd integer satisfying $ab>1+\frac{3}{2}\pi$, is nowhere differentiable 
 and H\"older continuous with exponent $-\frac{\ln(a)}{\ln(b)}\in(0,1)$ (in particular, it is continuous).
Later, many other researchers constructed such functions, among others Darboux, Dini and Takagi.
Poincar\'e was the first to call such functions the monsters of analysis.
Further, Shidfar and Sabetfakhri \cite{ShiSab} showed that the Takagi function is H\"older continuous with any exponent in $(0,1)$
(using the same technique that we use in the proof of part (ii) of Proposition \ref{Pro1}), but it is not 
H\"older continuous with exponent $1$, i.e., not Lipschitz continuous.

Mishura and Schied \cite{MisSch} studied the continuous $p^{\mathrm{th}}$-variation function of a signed Takagi-Landsberg function
 $g^{(H)}$ with Hurst parameter $H\in(0,1)$ along the sequence of dyadic partitions $\{k2^{-n}: k=0,1,\ldots,2^n\}$, $n\in\NN$,
 where $g^{(H)}:[0,1]\to\RR$ is defined by
 \begin{align}\label{signed_TL}
   g^{(H)}(t):=\sum_{m=0}^\infty 2^{m\left(\frac{1}{2}-H\right)} \sum_{k=0}^{2^m-1} \theta_{m,k}e_{m,k}(t),\qquad t\in[0,1],
 \end{align}
 where $\theta_{m,k}\in\{-1,1\}$ are arbitrary, and $e_{m,k}$ are the so-called Faber-Schauder functions given by
 \[
  e_{0,0}(t):=(\min(t,1-t))^+ \qquad \text{and}\qquad
  e_{m,k}(t):=2^{-\frac{m}{2}}e_{0,0}(2^m t-k), \qquad t\in\RR.
 \]
Note that different choices of $\theta_{m,k}$, $m\in\ZZ_+$, $k\in\{0,1,\ldots,2^m-1\}$, may result in
 different functions $g^{(H)}$, nonetheless, we do not denote the dependence of $g^{(H)}$ on $\theta_{m,k}$.
One can check that the series in the definition \eqref{signed_TL} of $g^{(H)}$ converges uniformly on $[0,1]$ for all $H\in(0,1)$
 and all possible choices of $\theta_{m,k}\in\{-1,1\}$ (see page 260 in Mishura and Schied \cite{MisSch}).
The notion of signed Takagi-Landsberg functions is a natural generalization of the well-known Takagi function,
 which formally corresponds to the case $H=1$ and $\theta_{m,k}=1$ for all $m\in\ZZ_+$, $k\in\{0,1,\ldots,2^m-1\}$.
Turning back to the case $H\in(0,1)$, Mishura and Schied \cite[Theorem 2.1]{MisSch} showed that
 \begin{align*}
 \lim_{n\to\infty} \sum_{k=0}^{\lfloor t2^n\rfloor} \vert g^{(H)}((k+1)2^{-n}) - g^{(H)}(k2^{-n})\vert^p
      = \begin{cases}
         0 & \text{if $p>\frac{1}{H}$,}\\
         t\cdot \EE(\vert Z_H\vert^p) & \text{if $p=\frac{1}{H}$,}\\
          \infty & \text{if $1\leq p<\frac{1}{H}$}
        \end{cases}
 \end{align*}
 for all $t\in(0,1]$, where $Z_H := \sum_{m=0}^\infty 2^{m(H-1)} Y_m$ with a sequence of independent and identically distributed 
 random variables $(Y_m)_{m\in\ZZ_+}$ such that $\PP(Y_0=1) = \PP(Y_0=-1) = \frac{1}{2}$.
The distribution of $Z_H$ is called the distribution of the infinite (symmetric) Bernoulli convolution
 with parameter $2^{H-1}$ (see Remark 2.2 in Mishura and Schied \cite{MisSch}).
As a consequence, taking into account that $g^{(H)}(0)=0$ and $g^{(H)}$ is continuous, 
 the continuous $(\frac{1}{H})^{\mathrm{th}}$-variation function
 of $g^{(H)}$ along the sequence of dyadic partitions $\{k2^{-n}: k=0,1,\ldots,2^n\}$, $n\in\NN$, takes the form
 $\langle g^{(H)}\rangle^{(\frac{1}{H})}(t) = t \EE(\vert Z_H\vert^{\frac{1}{H}})$, $t\in[0,1]$.
If $\frac{1}{H}$ is an even integer, then Escribano et al.\ \cite[Theorem 1]{EscSasTor} derived an explicit formula for
 $\EE(\vert Z_H\vert^{\frac{1}{H}})$ in terms of Bernoulli numbers.

Recently, Schied and Zhang \cite{SchZha1} have studied the continuous $p^{\mathrm{th}}$-variation function of a function $f:[0,1]\to\RR$
 along the sequence of $b$-adic partitions $\{kb^{-n}: k=0,1,\ldots,b^n\}$, $n\in\NN$,
 where $p\geq 1$, $b\in\NN\setminus \{1\}$, and
 $f$ is defined by
 \begin{align}\label{help3}
   f(t):=\sum_{m=0}^\infty \beta^m \phi(b^m t),\qquad t\in[0,1],
 \end{align}
 where $\beta\in(-1,1)\setminus\{0\}$ and $\phi:\RR\to\RR$ is a periodic function with period $1$, Lipschitz continuous and vanishes on the set of integers $\ZZ$.
Note that if $\phi:\RR\to\RR$, $\phi(t):= \nu \sin(2\pi t) + \varrho\cos(2\pi t) - \varrho$, $t\in\RR$, with some $\nu,\varrho\in\RR$,
 then $f$ is a Weierstrass function,
 and if $\phi:\RR\to\RR$, $\phi(t):=\min_{z\in\ZZ}\vert t-z\vert$, $t\in\RR$ (i.e., $\phi(t)$ is the distance of $t$ to the nearest integer),
 $b:=2$ and $\beta:=\frac{1}{2}$, then $f$ is the Takagi function.
In case of the Takagi function, the corresponding function $\phi$ is called a triangular wave function, which is Lipschitz continuous.
Note also that the triangular wave function and the Faber-Schauder function $e_{0,0}$ coincide on the interval $[0,1]$.
One can check that the representation of the Takagi function using the triangular wave function
 coincides with the previously mentioned one, which uses Faber-Schauder functions.
Motivated by these special cases, one can call $f$ defined by \eqref{help3} a Weierstrass-type function.
Schied and Zhang \cite[Theorem 2.1]{SchZha1}, among others, showed that if $\vert\beta\vert<\frac{1}{b}$, then $f$ is of bounded variation;
 if $\vert\beta\vert=\frac{1}{b}$, then, for $p>1$, we have
 \[
  \lim_{n\to\infty} \sum_{k=0}^{\lfloor tb^n\rfloor} \vert f((k+1)b^{-n}) - f(kb^{-n})\vert^p =0,\qquad t\in[0,1];
 \]
 and if $\frac{1}{b}<\vert \beta\vert < 1$, then there exists an appropriately defined random variable $Z$ such that
 $f$ is of bounded variation if and only if $\PP(Z=0)=1$, and otherwise (i.e., if $\PP(Z\ne 0)>0$) we have
 \[
  \lim_{n\to\infty} \sum_{k=0}^{\lfloor tb^n\rfloor} \vert f((k+1)b^{-n}) - f(kb^{-n})\vert^p
      = \begin{cases}
         0 & \text{if $p>q$,}\\
         t\cdot \EE(\vert Z\vert^q) & \text{if $p=q$,}\\
          \infty & \text{if $1\leq p<q$}
        \end{cases}
 \]
 for all $t\in(0,1]$, where $q:=-\log_{\vert\beta\vert}(b)\in(1,\infty)$.
As a consequence, taking into account that $f(0)=0$ and $f$ is continuous, if $\frac{1}{b}<\vert \beta\vert <1$ and $f$ is not of bounded variation,
 then the continuous $q^{\mathrm{th}}=\big(-\log_{\vert\beta\vert}(b)\big)^{\mathrm{th}}$-variation function of $f$ along the sequence
 of $b$-adic partitions $\{kb^{-n}: k=0,1,\ldots,b^n\}$, $n\in\NN$, takes the form
 $\langle f\rangle^{(q)}(t) = t \EE(\vert Z\vert^q)$, $t\in[0,1]$, where $\PP(Z=0)<1$.
We also mention that in Appendix A of Schied and Zhang \cite{SchZha2}, one can find an extension of Theorem 2.1 in Schied and Zhang \cite{SchZha1},
 where instead of the Lipschitz continuity of $\phi$, one assumes its H\"older continuity.
This result is interesting on its own right, but also plays a crucial role in the proofs in \cite{SchZha2}.

In this paper, we generalize the above recalled Theorem 2.1 of Schied and Zhang \cite{SchZha1}
 and the results in Appendix A of Schied and Zhang \cite{SchZha2}.
Namely, we study the continuous $p^{\mathrm{th}}$-variation function of a function $f:[0,1]\to\RR$
 along the sequence of $b$-adic partitions $\{kb^{-n}: k=0,1,\ldots,b^n\}$, $n\in\NN$, where $p\geq 1$, $b\in\NN\setminus\{1\}$
 and $f$ is defined by
 \begin{align}\label{help7}
   f(t):=\sum_{m=0}^\infty \xi_m \psi(b^{-m})\phi(b^m t),\qquad t\in[0,1],
 \end{align}
 where $\xi_m$, $m\in\ZZ_+$, is a (bounded) sequence of real numbers such that $\vert \xi_m\vert\leq L$, $m\in\ZZ_+$,
  with some $L\in\RR_{++}$, $\phi:\RR\to\RR$ is a periodic function with period $1$,
 H\"older continuous with some exponent $\gamma\in(0,1]$, and vanishes on $\ZZ$,
 and $\psi:\RR_{++}\to \RR_{++}$ is a submultiplicative function such that $\psi(b^{-1})\in(0,1)$.
To highlight a connection between submultiplicative functions and power functions, in the Appendix,
 we recall a result on the decomposition of submultiplicative functions
 in terms of the product of a power function and another appropriate function
 due to Finol and Maligranda \cite[Theorem 1]{FinMal}, and we also provide some non-trivial examples of submultiplicative functions.
The exclusion of the case $b=1$ in the definition \eqref{help7} of $f$ is natural, since in this case
 $f(t)$ should be defined as $\left( \sum_{m=0}^\infty \xi_m \right) \psi(1)\phi(t)$, $t\in[0,1]$,
 however, the series $\sum_{m=0}^\infty \xi_m$ does not converge in general.
Note that if $\psi$ is multiplicative, then $\psi(b^{-m}) = (\psi(b^{-1}))^m$, $m\in \ZZ_+$.
If, in addition, $\gamma=1$ (the Lipschitz continuous case), $\xi_m=1$ for all $m\in\ZZ_+$, or $\xi_m=(-1)^m$ for all $m\in\ZZ_+$, then
 we get back the form \eqref{help3} of $f$ by setting $\beta:=\psi(b^{-1})\in(0,1)$, and $\beta:=-\psi(b^{-1})\in(-1,0)$, respectively.
Motivated by this, one can call $f$ defined by \eqref{help7} a Weierstrass-type function as well.
We call the attention that the function $f$ defined by either \eqref{help3} or \eqref{help7} also depends on the parameter $b$
 of the sequence of $b$-adic partitions along which the continuous $p^{\mathrm{th}}$-variation function of $f$ will be investigated.
Note also that if $\psi:\RR_{++}\to\RR_{++}$, $\psi(x):=x^{-\log_2(a)}$, $x\in\RR_{++}$ (where $a\in(0,1)$),
 $\phi:\RR\to\RR$, $\phi(t):=2\min_{z\in\ZZ}\vert t-z\vert$, $t\in\RR$
 (i.e., $\phi(t)$ is two-times the distance of $t$ to the nearest integer), $b:=2$
 and $\xi_m$, $m\in\ZZ_+$, are independent and identically distributed random variables
 taking values in $\{-1,+1\}$, then $f$ is the random Takagi function considered
 by Allaart \cite[Section 8]{All2008}, who, among others, investigated
 the distribution of the maxima of such functions.
Furthermore, if $\xi_m\in\{-1,+1\}$, $m\in\ZZ_+$, are arbitrary, $\psi:\RR_{++}\to\RR_{++}$, $\psi(x):=x$, $x\in\RR_{++}$,
 $\phi:\RR\to\RR$, $\phi(t):=\min_{z\in\ZZ}\vert t-z\vert$, $t\in\RR$, and $b:=2$,
 then $f$ is the signed Takagi function considered by Allaart \cite[formula (2)]{All2013},
 who investigated the level sets of such functions.

The paper is structured as follows.
Section \ref{Sec_Wei_type_frac_func} is devoted to studying H\"older continuity of the Weierstrass-type function
 $f$ defined in \eqref{help7}.
We can distinguish three cases according to $\psi(b^{-1}) < b^{-\gamma}$, $\psi(b^{-1}) = b^{-\gamma}$, and $\psi(b^{-1}) > b^{-\gamma}$,
 and the H\"older continuity of $f$ is proved in these three cases, see Proposition \ref{Pro1}.
In Section \ref{Sec_Wei_type_frac_func_new}, we investigate $p^\mathrm{th}$-variation functions of the Weierstrass-type function
 $f$ defined in \eqref{help7}.
In Lemma \ref{Lem_p_var_expression}, we derive a probabilistic representation of $V^{p,t}_n(f)$,
 where $p\geq 1$, $t\in[0,1]$, and $n\in\NN$.
Using it, in Theorems \ref{Thm3}, \ref{Thm4} and \ref{Thm5}, which correspond to the cases $\psi(b^{-1}) < b^{-\gamma}$,
 $\psi(b^{-1}) = b^{-\gamma}$, and $\psi(b^{-1}) > b^{-\gamma}$,
 we investigate the asymptotic behavior of $V^{p,t}_n(f)$ as $n\to\infty$, where $t\in[0,1]$.
In Theorems \ref{Thm3} and \ref{Thm4}, we study the cases
 $p>\frac{1}{\gamma}$ and $p=\frac{1}{\gamma}$ separately,
 while, in Theorem \ref{Thm5}, the cases $p>\frac{1}{r}$ and $p=\frac{1}{r}$ are studied separately,
 where $r:=-\log_{b}(\psi(b^{-1}))\in(0,\gamma)$.
In the cases $p=\frac{1}{\gamma}$ and $p=\frac{1}{r}$ in question, we can only prove the boundedness of
 limit superior of $V_n^{p,t}(f)$ as $n\to\infty$, where $t\in[0,1]$
 (in case of Theorem \ref{Thm4} an appropriate normalization also comes into play),
 however, these types of results (in some sense) fit the assumption of the submultiplicativity of $\psi$.
The remaining cases $p<\frac{1}{\gamma}$ and $p<\frac{1}{r}$ are considered in Theorem \ref{Thm_last_parts}, where 
 the three regions $\psi(b^{-1}) < b^{-\gamma}$, $\psi(b^{-1}) = b^{-\gamma}$ and $\psi(b^{-1}) > b^{-\gamma}$ 
 are handled together.
Note that the constant $r$ corresponds to $\frac{1}{q}$ in the setup of Schied and Zhang \cite{SchZha1} with the replacement of $\beta$
 by $\psi(b^{-1})$.
Furthermore, we mention that, in view of Remark \ref{Rem4}, the subclass of those continuous functions $g:[0,1]\to\RR$
 for which $\limsup_{n\to\infty}V_n^{p,1}(g)<\infty$ is also an important one due to the work of Das and Kim \cite{DasKim}.
Corollary \ref{Cor_Lip} is about the special case $\psi(b^{-1}) < b^{-1}$ and $\gamma=1$ (Lipschitz continuity),
 when it turns out that $f$ is Lipschitz continuous and of bounded variation.
In Section \ref{Sec_Sch_Zha}, we improve our results
 in Section \ref{Sec_Wei_type_frac_func} for the case
 $\psi(b^{-1}) > b^{-\gamma}$ with a multiplicative function $\psi$.
This case was considered in Schied and Zhang \cite[part (iii) of Theorem 2.1]{SchZha1}
 and \cite[Proposition A.2]{SchZha2}, and it turns out that we can also improve their results in question,
 see Propositions \ref{Pro7} and \ref{Pro_Z} and Corollaries \ref{Cor_bounded_var} and \ref{Cor_Z}.
In Section \ref{Sec_Riesz_var}, we prove the finiteness of the limit superior of appropriately normalized Riesz variations of the Weierstrass-type function
 given in \eqref{help7} along the sequence of $b$-adic partitions, see Theorem \ref{Thm_Riesz}.
We also specialize our results for Riesz variations
 to the case $\gamma=1$ (Lipschitz continuity) and $\psi(b^{-1})=b^{-1}$, where $\psi$ is multiplicative,
 see the paragraph after the proof of Theorem \ref{Thm_Riesz}.
We close the paper with an Appendix in which we recall a decomposition of submultiplicative functions
 due to Finol and Maligranda \cite[Theorem 1]{FinMal}, and we also provide some non-trivial examples of submultiplicative functions.

Finally, we summarize the novelties of the paper.
We emphasize that Schied and Zhang \cite[Proposition A.2]{SchZha2}
 determined the continuous $p^{\mathrm{th}}$-variation function of $f$ defined by \eqref{help7} 
 with a H\"older continuous function $\phi$ with index $\gamma\in(0,1]$ along the sequence of $b$-adic partitions
 only in the case when $\xi_m=1$, $m\in\ZZ_+$,
 $\psi$ is multiplicative satisfying $\psi(b^{-1})> b^{-\gamma}$ and $p=\frac{1}{r}$, where $r=-\log_b(\psi(b^{-1}))\in(0,\gamma)$.
In our paper, we derive some estimates also in the cases $\psi(b^{-1}) < b^{-\gamma}$ and $\psi(b^{-1})=b^{-\gamma}$ for a submultiplicative function $\psi$, 
 bounded coefficients $\xi_{m}$, $m\in\ZZ_{+}$,
 and we do not restrict ourselves to $p=\frac{1}{r}$.
We point out that, in the case $\psi(b^{-1}) > b^{-\gamma}$, the random variables $T_n$, $n\in\NN$, defined in \eqref{help_Tn}
 play a crucial role in Theorem \ref{Thm5}, and these random variables are appropriate extensions of the random variables $Z_n$, $n\in\NN$,
 defined in \eqref{help34}, appearing also in the proof of part (iii) of Theorem 2.1 in  Schied and Zhang \cite{SchZha1}.
However, in order to be able to handle the case $\psi(b^{-1}) < b^{-\gamma}$ in Theorem \ref{Thm3}, 
 we had to introduce some new random variables $W_n$, $n\in\NN$, see \eqref{help_Wn}.
We also emphasize that the random variables $T_n$ and $W_n$, $n\in\NN$, coincide in the case $\psi(b^{-1}) = b^{-\gamma}$,
 studied in Theorem \ref{Thm4}.
Investigation of Riesz variations of Weierstrass-type functions along the sequence of $b$-adic partitions is also a new feature,
 which may call the attention to the fact that other types of variations rather than the $p^{\mathrm{th}}$-variation can be interesting.

\section{H\"older continuity of Weier\-strass-type functions}\label{Sec_Wei_type_frac_func}

Our first result states that the function $f$ given by \eqref{help7} is well-defined and continuous.
Recall that $\xi_m$, $m\in\ZZ_+$, is a (bounded) sequence of real numbers such that $\vert \xi_m\vert\leq L$, $m\in\ZZ_+$, with some $L\in\RR_{++}$.

\begin{Lem}\label{Lem_f_welldefined}
The series in \eqref{help7} converges absolutely and uniformly on $[0,1]$,
 and consequently, the function $f$ given by \eqref{help7} is well-defined and continuous.
\end{Lem}

\begin{proof}
First, note that $\phi$ is a continuous and periodic function, and hence it is bounded.
Therefore, using that $\psi$ is non-negative, submultiplicative and $\psi(b^{-1})\in(0,1)$, we have that
 \[
   \vert f(t)\vert \leq K L \sum_{m=0}^\infty \psi(b^{-m})
      \leq K L \sum_{m=0}^\infty (\psi(b^{-1}))^{m}
      = \frac{KL}{1-\psi(b^{-1})}
      <\infty,\qquad t\in[0,1],
 \]
 with some constant $K\in\RR_{++}$ satisfying $\vert \phi(t)\vert\leq K$, $t\in\RR$.
Consequently, due to the Weierstrass M-test, the series in \eqref{help7} converges absolutely and uniformly on $[0,1]$,
 so the function $f$ given by \eqref{help7} is well-defined.
The uniform limit theorem implies that $f$ is continuous as well.
\end{proof}

Note that, since $\phi$ is periodic with period $1$, the function $f$ defined by \eqref{help7}
 could be extended to $\RR$ as a periodic function with period $1$ satisfying $f(t)=0$, $t\in\ZZ$.

Our aim is to investigate the existence of a continuous $p^\mathrm{th}$-variation function of $f$ given by \eqref{help7} along the sequence of $b$-adic partitions,
 where $p\geq 1$ and $b\in\NN\setminus\{1\}$.
Since $\phi$ is H\"older continuous with exponent $\gamma\in(0,1]$, there exists a constant $C\in\RR_{++}$ such that
 \begin{align}\label{phi_Holder}
    \vert \phi(x) - \phi(y)\vert \leq C \vert x - y\vert^\gamma, \qquad x,y\in\RR.
 \end{align}

Our first result is a counterpart of Proposition A.1 in Schied and Zhang \cite{SchZha2}.
In our setup $\phi$ is submultiplicative, while in Schied and Zhang \cite{SchZha2} $\phi$ is multiplicative,
 and we present our result in a somewhat different form.

In what follows, let 
 \begin{align}\label{beta_def}
                 r:=-\log_b(\psi(b^{-1}))\in (0,\infty).
 \end{align}
              
\begin{Pro}\label{Pro1}
Let us consider the function $f$ defined by \eqref{help7}, and recall that $\gamma\in(0,1]$ is the exponent of H\"older continuity
 for the function $\phi$.
 \begin{itemize}
  \item[(i)] If $\psi(b^{-1}) < b^{-\gamma}$, then $f$ is H\"older continuous with exponent $\gamma$.
  \item[(ii)] If $\psi(b^{-1}) > b^{-\gamma}$, then $f$ is H\"older continuous with exponent $r\in (0,\gamma)$, where $r$ is given by \eqref{beta_def}.
  \item[(iii)] If $\psi(b^{-1}) = b^{-\gamma}$, then, for all $\eta>0$, there exists a constant $C_1\in\RR_{++}$ such that
              \begin{align}\label{help12}
               \vert f(t) - f(s)\vert \leq C_1 \vert t-s\vert^\gamma \log_b(\vert t-s\vert^{-1}) \quad \text{for $s,t\in[0,1]$ with $0<\vert s-t \vert \leq \frac{1}{1+\eta}$.}
             \end{align}
             Furthermore, $f$ is H\"older continuous with any positive exponent strictly less than $\gamma$.
 \end{itemize}
\end{Pro}

The proof of Proposition \ref{Pro1} shows that, for all $\eta>0$, the constant $C_1\in\RR_{++}$ in part (iii) of Proposition \ref{Pro1} can be chosen as 
 \begin{align}\label{constant_C1}
  L\left(  1+ \frac{1}{\log_b(1+\eta)} \right) \left(  C + 2\sup_{x\in\RR} \vert \phi(x)\vert \frac{1}{1 - \psi(b^{-1})} \right),
 \end{align}
 where the constant $C\in\RR_{++}$ is given by \eqref{phi_Holder}.
We note that Proposition \ref{Pro1} can be interpreted as follows:
in case of (i) the H\"older exponent of $f$ is the same as that of $\phi$,
 i.e., $f$ fluctuates as  ``roughly'' as $\phi$;
in case of (ii) the H\"older exponent of $f$ is strictly less than that of $\phi$,
 i.e, $f$ fluctuates more  ``roughly'' than $\phi$;
and in case of (iii) the H\"older exponent of $f$ is strictly less than that of $\phi$,
 but it can be arbitrarily close to that.

\noindent{\bf Proof of Proposition \ref{Pro1}.}
(i): Let us suppose that $\psi(b^{-1}) < b^{-\gamma}$.
Using \eqref{phi_Holder} and the nonnegativity and submultiplicativity of $\psi$, for all $s,t\in[0,1]$, we have
 \begin{align}\label{help18}
  \begin{split}
  \vert f(t) - f(s)\vert
     &\leq L \sum_{m=0}^\infty \psi(b^{-m}) \vert \phi(b^m t) - \phi(b^m s) \vert
       \leq L C \sum_{m=0}^\infty (\psi(b^{-1}))^m \vert b^m t - b^m s \vert^\gamma \\
     & =  L C \left( \sum_{m=0}^\infty (\psi(b^{-1})b^\gamma)^m  \right)  \vert t - s \vert^\gamma
       = \frac{LC}{1-\psi(b^{-1})b^\gamma} \vert t-s \vert^\gamma,
   \end{split}
 \end{align}
 where $C\in\RR_{++}$ is given by \eqref{phi_Holder}.

(ii): Let us suppose that $\psi(b^{-1}) > b^{-\gamma}$.
Let $s\ne t$, $s,t\in[0,1]$, be arbitrarily fixed.
Then one can choose an $N\in\NN$ (depending on $\vert t-s\vert$) such that $b^{-N} < \vert t-s\vert \leq b^{-(N-1)}$.
Similarly as in case (i), using also that $\phi$ is bounded (checked in the proof of Lemma \ref{Lem_f_welldefined}),
 we have
 \begin{align}\label{help9}
  \begin{split}
  \vert f(t) - f(s)\vert
     &\leq L \sum_{m=0}^{N-1} \psi(b^{-m}) \vert \phi(b^m t) - \phi(b^m s) \vert
           + L \sum_{m=N}^\infty \psi(b^{-m}) \vert \phi(b^m t) - \phi(b^m s) \vert\\
     &\leq L C \left( \sum_{m=0}^{N-1} (\psi(b^{-1})b^\gamma)^m  \right)  \vert t - s \vert^\gamma
           + 2L \sup_{x\in\RR} \vert \phi(x)\vert \sum_{m=N}^\infty (\psi(b^{-1}))^m \\
     & =LC \frac{(\psi(b^{-1})b^\gamma)^N - 1}{\psi(b^{-1})b^\gamma - 1} \vert t - s \vert^\gamma
         + 2 L \sup_{x\in\RR} \vert \phi(x)\vert \frac{ (\psi(b^{-1}))^N }{1 - \psi(b^{-1}) } \\
     & \leq LC \frac{(\psi(b^{-1}))^N b^{\gamma N}}{\psi(b^{-1})b^\gamma - 1} \vert t - s \vert^\gamma
             + 2L \sup_{x\in\RR} \vert \phi(x)\vert \frac{(\psi(b^{-1}))^N}{1 - \psi(b^{-1})}.
  \end{split}
 \end{align}
Here, by the choice of $N$ and using that $r=-\log_b(\psi(b^{-1}))>0$ (due to $b\in\NN\setminus\{1\}$ and $\psi(b^{-1})\in(0,1)$),
 we have
 \begin{align*}
  (\psi(b^{-1}))^N = b^{N\log_b(\psi(b^{-1}))} =  b^{-Nr}
                   < \vert t - s \vert^r.
 \end{align*}
Further, $\vert t-s \vert \leq b^{1-N}$ implies that $b^{\gamma N}\leq b^\gamma \vert t - s \vert^{-\gamma}$.
As a consequence, using also \eqref{help9}, we have
 \begin{align}\label{help19}
  \vert f(t) - f(s)\vert
   \leq L \left( \frac{C b^\gamma}{\psi(b^{-1})b^\gamma - 1}
        + \sup_{x\in\RR} \vert \phi(x)\vert \frac{2}{1 - \psi(b^{-1})} \right)
          \vert t - s \vert^r,
 \end{align}
 as desired.
If $s,t\in[0,1]$ are such that $s=t$, then \eqref{help19} readily holds.
The inequality $0< r < \gamma$ is a consequence of $b^{-\gamma} < \psi(b^{-1}) \in(0,1)$ and $b\in\NN\setminus\{1\}$.

(iii): Let us suppose that $\psi(b^{-1}) = b^{-\gamma}$.
Let $s\ne t$, $s,t\in[0,1]$, be arbitrarily fixed.
Then one can choose an $N\in\NN$ (depending on $\vert t-s\vert$) such that $b^{-N} < \vert t-s\vert \leq b^{-(N-1)}$.
Similarly as \eqref{help9}, we get
 \begin{align}\label{help10}
  \begin{split}
  \vert f(t) - f(s)\vert
     &\leq L \sum_{m=0}^{N-1} \psi(b^{-m}) \vert \phi(b^m t) - \phi(b^m s) \vert
           + L \sum_{m=N}^\infty \psi(b^{-m}) \vert \phi(b^m t) - \phi(b^m s) \vert\\
     &\leq L C \left( \sum_{m=0}^{N-1} (\psi(b^{-1})b^\gamma)^m  \right)  \vert t - s \vert^\gamma
           + 2 L \sup_{x\in\RR} \vert \phi(x)\vert \sum_{m=N}^\infty (\psi(b^{-1}))^m \\
     &= L C N \vert t - s \vert^\gamma
           + 2 L \sup_{x\in\RR} \vert \phi(x)\vert \frac{(\psi(b^{-1}))^N}{1-\psi(b^{-1})}.
  \end{split}
 \end{align}
Here, by the choice of $N$, we have $b^{N-1}\leq \vert t-s\vert^{-1}$, and hence $N-1\leq \log_b(\vert t-s\vert^{-1})$,
 i.e., $N\leq 1 + \log_b(\vert t-s\vert^{-1})$.
The choice of $N$ also yields that $(\psi(b^{-1}))^N = (b^{-\gamma})^N = b^{-\gamma N} < \vert t-s\vert^\gamma$.
Consequently, using \eqref{help10}, we obtain that
 \begin{align}\label{help11}
  \begin{split}
   \vert f(t) - f(s)\vert
      \leq L C \big(1 + \log_b(\vert t-s\vert^{-1}) \big)\cdot \vert t - s \vert^\gamma
           + 2L \sup_{x\in\RR} \vert \phi(x)\vert \frac{1}{1-\psi(b^{-1})}\cdot \vert t - s \vert^\gamma .
 \end{split}
 \end{align}
We check that if $\eta>0$ and $\vert s-t\vert\leq \frac{1}{1+\eta}$ hold as well, then there exists a constant $C_1\in\RR_{++}$ such that
 $\vert f(t) - f(s)\vert \leq L C_1 \vert t-s\vert^\gamma \log_b(\vert t-s\vert^{-1})$ holds, which yields \eqref{help12}.
Taking into account \eqref{help11}, for this, it is enough to check that there exists a constant $C^{*}\in\RR_{++}$ such that
 \[
   1 \leq 1 + \log_b(\vert u-v\vert^{-1}) \leq C^* \log_b(\vert u-v\vert^{-1})
 \]
 for all $u,v\in[0,1]$ satisfying $0<\vert u-v\vert\leq \frac{1}{1+\eta}$.
These two inequalities can be checked as follows:
 $\vert u-v\vert\leq \frac{1}{1+\eta}$ and $b\in\NN\setminus\{1\}$ imply that $0<\log_b(1+\eta)\leq \log_b(\vert u-v\vert^{-1})$
 (yielding the first inequality), and hence
 $1 + \log_b(\vert u-v\vert^{-1}) \leq \big( \frac{1}{\log_b(1+\eta)} + 1 \big)\log_b(\vert u-v\vert^{-1})$,
 yielding the second inequality with $C^*:=\frac{1}{\log_b(1+\eta)}+1$.  
Therefore, the constant $C_1$ can be chosen as in \eqref{constant_C1}.
 
Now, we turn to prove that $f$ is H\"older continuous with any positive exponent strictly less than $\gamma$.
First, we check that, for all $\vare\in(0,\gamma)$  and $\eta>0$, there exists a constant $C_2\in\RR_{++}$ such that
 \begin{align}\label{help13}
   \vert f(t) - f(s)\vert \leq C_2  \vert t-s\vert^{\gamma-\vare} \quad \text{for $s,t\in[0,1]$ with $0<\vert s-t \vert \leq \frac{1}{1+\eta}$.}
 \end{align}
For all $\vare>0$, we have that
 \[
   \sup_{x\in(0,1)} x^\vare\log_b(x^{-1})  <\infty,
 \]
 since, by L'Hospital's rule,
 \[
   \lim_{x\downarrow 0} x^\vare \log_b(x^{-1})
      = -\lim_{x\downarrow 0} \frac{\log_b(x)}{x^{-\vare}}
      = -\lim_{x\downarrow 0} \frac{\frac{1}{\ln(b)}\cdot \frac{1}{x}}{-\vare x^{-\vare-1}}
      = \frac{1}{\vare\ln(b)} \lim_{x\downarrow 0} x^\vare = 0.
 \]
Hence, for all $\vare\in(0,\gamma)$, $\eta>0$, and $s,t\in[0,1]$ with $0<\vert s-t\vert\leq \frac{1}{1+\eta}$, using \eqref{help12}, we get
 \begin{align*}
   \vert f(t) - f(s)\vert
      \leq C_1 \left( \sup_{x\in(0,1)} x^\vare\log_b(x^{-1}) \right) \vert t-s\vert^{\gamma-\vare},
 \end{align*}
 yielding \eqref{help13} by choosing $C_2:=C_1 \sup_{x\in(0,1)} x^\vare\log_b(x^{-1})$.
Next, we verify that for all $\vare\in(0,\gamma)$, there exists a constant $C_3\in\RR_{++}$ such that
 \begin{align}\label{help48}
  \vert f(t) - f(s)\vert \leq C_3 \vert t-s\vert^{\gamma-\vare}, \qquad s,t\in[0,1],
 \end{align}
 i.e., $f$ is H\"older continuous with exponent $\gamma-\vare$.
We check \eqref{help48} separating the cases $\vert t-s\vert=0$, $\vert t-s\vert\in(0,\frac{1}{2}]$, and $\vert t-s\vert\in(\frac{1}{2},1]$.  
If $s,t\in[0,1]$ are such that $\vert s-t\vert=0$, i.e., $s=t$, then the inequality \eqref{help48} holds trivially.
If $s,t\in[0,1]$ are such that $\vert s-t\vert\in(0,\frac{1}{2}]$, then  \eqref{help13} with $\eta:=1$ yields that 
 \eqref{help48} holds with $C_3:=C_2$. 
Suppose now that $s,t\in[0,1]$ are such that $\vert s-t\vert\in( \frac{1}{2},1]$.
By Lemma \ref{Lem_f_welldefined}, $f$ is bounded, i.e., there exists an $M>0$ such that $\vert f(u)\vert<M$, $u\in[0,1]$,
 yielding that $\vert f(u) - f(v)\vert\leq \vert f(u)\vert + \vert f(v) \vert < 2M$, $u,v\in[0,1]$.
Since $\gamma-\vare>0$, we have that $2^{-(\gamma-\vare)}<\vert s-t\vert^{\gamma-\vare}$, yielding that 
 $2M<M 2^{\gamma-\vare+1}\vert s-t\vert^{\gamma-\vare}$.
Hence, if $s,t\in[0,1]$ are such that $\vert s-t\vert\in( \frac{1}{2},1]$, then 
 \[
    \vert f(t) - f(s)\vert \leq M 2^{\gamma-\vare+1}\vert s-t\vert^{\gamma-\vare},
 \] 
 i.e., \eqref{help48} holds with $C_3:=M 2^{\gamma-\vare+1}$.
All in all, if $s,t\in[0,1]$, then \eqref{help48} holds with $C_3:=\max(C_2, M 2^{\gamma-\vare+1})$, as desired.
\proofend

Next, we give an example which shows that, in part (iii) of Proposition \ref{Pro1}, 
 in general, it does not hold that $f$ is H\"older continuous with exponent $\gamma$.
 
\begin{Ex}
Let $b\in\NN\setminus\{1\}$, $\gamma\in(0,1]$, $\xi_m:=1$, $m\in\ZZ_+$, 
 $\phi:\RR\to\RR$, $\phi(t):=(\min_{z\in\ZZ} \vert t-z\vert)^\gamma$, $t\in\RR$,
 and $\psi:\RR_{++}\to\RR_{++}$, $\psi(x):=x^\gamma$, $x\in\RR_{++}$. 
Since $\RR\ni t\mapsto \min_{z\in\ZZ} \vert t-z\vert$ is Lipschitz continuous and $\gamma\in(0,1]$,
 it is known that $\phi$ is H\"older continuous with exponent $\min(\gamma,1)=\gamma$.
Then $\psi(b^{-1})=b^{-\gamma}$, and the function $f$ defined in \eqref{help7} takes the form 
 \[
    f(t)= \sum_{m=0}^\infty b^{-m\gamma} \phi(b^m t), \qquad t\in[0,1].
 \]
On the one hand, by part (iii) of Proposition \ref{Pro1}, $f$ is H\"older continuous with any exponent strictly less than $\gamma$.
On the other hand, we check that $f$ is not H\"older continuous with exponent $\gamma$.
On the contrary, suppose that $f$ is H\"older continuous with exponent $\gamma$, i.e., there exists a constant $C^*>0$ such that 
 \begin{align}\label{help49}
    \vert f(t) - f(s)\vert \leq C^* \vert t-s\vert^\gamma, \qquad t,s\in[0,1].
 \end{align}   
Let $t_n:=b^{-n}$, $n\in\NN$.
Then, using that $b^{m-n}\in(0,\frac{1}{2}]$ for $m\in\{0,\ldots,n-1\}$, $n\in\NN$, and $b^{m-n}\in\ZZ$ for $m\geq n$, $m,n\in\NN$,
 together with the fact that $\phi$ vanishes on $\ZZ$, we get that
 \begin{align}\label{help50}
  \begin{split}
  f(t_n)& = \sum_{m=0}^\infty b^{-m\gamma} \phi(b^{m-n})
         = \sum_{m=0}^{n-1} b^{-m\gamma} \phi(b^{m-n}) +  \sum_{m=n}^\infty b^{-m\gamma} \phi(b^{m-n})
         = \sum_{m=0}^{n-1} b^{-m\gamma} (b^{m-n})^\gamma \\
        & = \sum_{m=0}^{n-1} b^{-n\gamma} 
         = n b^{-n\gamma}  = n t_n^\gamma,\qquad n\in\NN.
  \end{split}         
 \end{align}
By choosing $t:=t_n$, $n\in\NN$, and $s:=0$ in \eqref{help49}, and taking into account that $f(0)=0$, we have that 
 \[
   \frac{\vert f(t_n)\vert}{t_n^\gamma}\leq C^*, \qquad n\in\NN.
 \] 
By \eqref{help50}, it implies that $n\leq C^*$, $n\in\NN$, leading us to a contradiction.
Consequently, $f$ is not H\"older continuous with exponent $\gamma$.
In the special case when $b=2$ and $\gamma=1$, we have that $f$ is the classical Takagi function, which 
 is not Lipschitz continuous, but H\"older continuous with any positive exponent strictly less than $1$
 (proved first by Shidfar and Sabetfakhri \cite{ShiSab}).
\proofend
\end{Ex} 
 
We remark that if $\xi_m=1$, $m\in\ZZ_+$, and $\psi$ is multiplicative in the definition \eqref{help7} of $f$, 
 then parts (i) and (ii) of Proposition \ref{Pro1} give back 
 part (a) of Proposition A.1 in Schied and Zhang \cite{SchZha2}.
Indeed, with the notations of Schied and Zhang \cite{SchZha2}, we have
 $\alpha=\psi(b^{-1})$ and the H\"older exponent takes the form
 \[
  K=(-\log_b(\psi(b^{-1})))\wedge \gamma
   =\begin{cases}
     \gamma & \text{if $\psi(b^{-1})< b^{-\gamma}$,}\\
     -\log_b(\psi(b^{-1})) = r & \text{if $\psi(b^{-1}) > b^{-\gamma}$.}
    \end{cases} 
 \]
Further, in the above situation, part (iii) of  Proposition \ref{Pro1}  coincides with what is proven 
 for part (b) of Proposition A.1 in Schied and Zhang \cite{SchZha2}.

In the following corollary, we consider the special case  $\gamma=1$ (i.e., $\phi$ is Lipschitz continuous) and $\psi(b^{-1}) < b^{-1}$,
 and we derive an upper bound for the total variation of $f$ as well.

\begin{Cor}\label{Cor_Lip}
Let us consider the function $f$ defined by \eqref{help7}.
Suppose that $\gamma=1$ (i.e., $\phi$ is Lipschitz continuous) and that $\psi(b^{-1}) < b^{-1}$.
Then $f$ is Lipschitz continuous,
 of bounded variation and its total variation is less than or equal to $\frac{LC}{1-\psi(b^{-1})b}$,
 where $C$ is given by \eqref{phi_Holder}.
\end{Cor}

\begin{proof}
By part (i) of Proposition \ref{Pro1}, we get that $f$ is H\"older continuous with exponent $\gamma=1$,
 i.e., $f$ is Lipschitz continuous.
Formula \eqref{help18} in the proof of part (i) of Proposition \ref{Pro1} also shows that
 \[
  \vert f(t) - f(s)\vert \leq \frac{LC}{1-\psi(b^{-1})b} \vert t-s\vert,\qquad s,t\in[0,1],
 \]
 where $C$ is given by \eqref{phi_Holder}.
This readily implies that $f$ is of bounded variation and its total variation is less than  or equal to
 $\frac{LC}{1-\psi(b^{-1})b}$, as desired.
\end{proof}

We also remark that if $\xi_m=1$, $m\in\ZZ_+$, and $\psi$ is multiplicative in the definition \eqref{help7} of $f$,
 then Corollary \ref{Cor_Lip} gives back part (a) of Theorem 2.1 in Schied and Zhang \cite{SchZha1}.

\section{$p^\mathrm{th}$-variation functions of Weier\-strass-type functions}\label{Sec_Wei_type_frac_func_new}

In this section, we investigate $p^\mathrm{th}$-variation functions of the Weierstrass-type functions
  $f$ defined in \eqref{help7}, similarly to what is developed in Schied and Zhang \cite{SchZha1} for the subclass of Weierstrass-type 
  functions defined in \eqref{help3}.

For each $m\in\NN$ and $k\in\ZZ_+$, let
 \begin{align}\label{help32}
  \lambda_{m,k}:=\frac{\phi((k+1)b^{-m}) - \phi(kb^{-m})}{b^{-m}},
 \end{align}
 which is the slope of the line connecting the points $(kb^{-m},\phi(kb^{-m}))$ and $((k+1)b^{-m},\phi((k+1)b^{-m}))$.
Since $\phi$ is periodic with period $1$, we have $\lambda_{m,k} = \lambda_{m,k+\ell b^m}$ for each $m\in\NN$ and $k,\ell\in\ZZ_+$.
Let $(U_n)_{n\in\NN}$ be a sequence of independent and identically distributed random variables such that
 $U_1$ is uniformly distributed on the finite set $\{0,1,\ldots,b-1\}$.
Further, for each $m\in\NN$, let us define the random variables
 \begin{align}\label{help33}
  R_m:=\sum_{i=1}^m U_i b^{i-1} \qquad \text{and}\qquad Y_m:=\lambda_{m,R_m}.
 \end{align}
One can check that $R_m$ is uniformly distributed on the set $\{0,\ldots,b^m-1\}$ for each $m\in\NN$.
Further, using \eqref{phi_Holder}, $\PP$-almost surely, we have
 \begin{align}\label{Y_estimate}
  \vert Y_m\vert\leq C b^{m(1-\gamma)}, \qquad m\in\NN,
 \end{align}
 where the constant $C\in\RR_{++}$ is given by \eqref{phi_Holder}.

Now we derive a probabilistic representation of $V^{p,1}_n(f)$, $p\geq 1$, $n\in\NN$, 
 in terms of $Y_m$, $m\in\ZZ_+$,
 where $f$ and $V^{p,1}_n(f)$ are defined in \eqref{help7} and \eqref{pth_var_Def}, respectively.
In the definition of $V^{p,1}_n(f)$, we consider the sequence of $b$-adic partitions corresponding to
 the same parameter $b\in\NN\setminus\{1\}$, which appears in the definition of $f$ in \eqref{help7}.
This setup is also assumed in Section \ref{Sec_Sch_Zha}.
The next result is in fact a generalization of Lemma 2.3 in Schied and Zhang \cite{SchZha1}.

\begin{Lem}\label{Lem_p_var_expression}
Let us consider the function $f$ defined by \eqref{help7}.
For each $n\in\NN$ and $p\geq 1$, we have
 \begin{align}\label{help14}
    V^{p,t}_n(f)
      = b^n \EE\left( \left\vert  \sum_{m=1}^n \xi_{n-m} \psi(b^{m-n}) b^{-m} Y_m \right\vert^p
                      \bone_{\{ R_n\leq \lfloor tb^n\rfloor\}} \right), 
     \qquad t\in[0,1],                  
 \end{align}
 where $V^{p,t}_n(f)$ and $(R_m)_{m\in\NN}$, $(Y_m)_{m\in\NN}$ are defined in \eqref{pth_var_Def} and \eqref{help33}, respectively.
If, in addition, $\psi$ is multiplicative as well, then, for each $n\in\NN$ and $p\geq 1$, we have
 \begin{align}\label{help15}
   V^{p,t}_n(f) = \big( (\psi(b^{-1}))^p b \big)^n 
                   \EE\left( \left\vert  \sum_{m=1}^n \xi_{n-m} (\psi(b^{-1})b )^{-m} Y_m \right\vert^p 
                             \bone_{\{ R_n\leq \lfloor tb^n\rfloor\}}
                             \right),
    \qquad  t\in[0,1].
 \end{align}
\end{Lem}

\begin{proof}
For each $n\in\NN$, let us consider the $n^{\mathrm{th}}$-truncation $f_n$ of $f$ given by
 \[
   f_n(t) := \sum_{m=0}^{n-1} \xi_m \psi(b^{-m})\phi(b^m t), \qquad t\in[0,1].
 \]
Then $f_n(kb^{-n}) = f(kb^{-n})$ for $n\in\NN$ and $k\in\{0,1,\ldots,b^n\}$, since $\phi(b^m k b^{-n}) = \phi(b^{m-n}k)=0$ for $m\in\{n,n+1,\ldots\}$
 due to the fact that $\phi$ vanishes on $\ZZ$.
Hence, using \eqref{help32}, for each $n\in\NN$ and $k\in\{0,1,\ldots,b^n-1\}$, we get
 \begin{align*}
  f((k+1)b^{-n}) - f(kb^{-n})
   & = f_n((k+1)b^{-n}) - f_n(kb^{-n}) \\
   & = \sum_{m=0}^{n-1} \xi_m \psi(b^{-m}) \big( \phi((k+1)b^{m-n}) - \phi(kb^{m-n}) \big) \\
   & = \sum_{m=0}^{n-1}  \xi_m \psi(b^{-m}) b^{m-n} \lambda_{n-m,k}.
 \end{align*}
Using \eqref{pth_var_Def}, part (i) of Remark \ref{Rem2} and \eqref{help33}, for each $n\in\NN$ it follows that
 \begin{align}\label{help17}
  \begin{split}
   V^{p,t}_n(f)
     & = \sum_{k=0}^{\lfloor tb^n\rfloor}\vert f((k+1)b^{-n}) - f(kb^{-n})\vert^p
       = \sum_{k=0}^{\lfloor tb^n\rfloor} \left\vert  \sum_{m=0}^{n-1} \xi_m \psi(b^{-m}) b^{m-n} \lambda_{n-m,k} \right\vert^p \\
     & = b^n \sum_{k=0}^{\lfloor tb^n\rfloor} \left[ \left\vert  \sum_{m=0}^{n-1} \xi_m \psi(b^{-m}) b^{m-n} \lambda_{n-m,k} \right\vert^p \PP(R_n=k)\right]\\
     & = b^n \EE \left( \left\vert  \sum_{m=0}^{n-1} \xi_m \psi(b^{-m}) b^{m-n} \lambda_{n-m,R_n} \right\vert^p  
                        \bone_{\{ R_n\leq \lfloor tb^n\rfloor\}}  \right) \\
     &  = b^n \EE \left( \left\vert  \sum_{\ell=1}^n \xi_{n-\ell} \psi(b^{\ell-n}) b^{-\ell} \lambda_{\ell,R_n} \right\vert^p  
                        \bone_{\{ R_n\leq \lfloor tb^n\rfloor\}}   \right),
                        \qquad t\in[0,1].
  \end{split}
 \end{align}
Using that $\phi$ is periodic with period 1, $U_i$  takes values in $\{ 0,1,\ldots,b-1\}$,
 and that $b^{i-1-\ell}\in\NN$ for $i>\ell$, $i,\ell\in\ZZ$,
 we have, for all $x\in\RR$ and each $n\in\NN$, $\ell\in\{1,\ldots,n\}$,
 the random variable $\sum_{i=\ell+1}^n U_i b^{i-1-\ell}$ takes values in $\ZZ_+$ and hence
 \begin{align}\label{help16}
  \phi(x+R_nb^{-\ell}) = \phi\left(x+ \sum_{i=1}^n U_i b^{i-1-\ell}\right)
                       = \phi\left(x+ \sum_{i=1}^\ell U_i b^{i-1-\ell}\right)
                       = \phi\left(x+ R_\ell b^{-\ell}\right).
 \end{align}
Using \eqref{help16} with $x:=b^{-\ell}$ and $x:=0$, respectively, for each $n\in\NN$ and $\ell\in\{1,\ldots,n\}$, we have
 \begin{align*}
  \lambda_{\ell,R_n}
    & = \frac{\phi((R_n+1)b^{-\ell}) - \phi(R_nb^{-\ell})}{b^{-\ell}}
      = \frac{\phi(b^{-\ell} + R_n b^{-\ell}) - \phi(R_n b^{-\ell})}{b^{-\ell}} \\
    &= \frac{\phi(b^{-\ell} + R_\ell b^{-\ell}) - \phi(R_\ell b^{-\ell})}{b^{-\ell}}
     = \lambda_{\ell,R_\ell} = Y_\ell.
 \end{align*}
This together with \eqref{help17} implies \eqref{help14}.

Now we turn to prove \eqref{help15}.
If, in addition, $\psi$ is multiplicative as well, then \eqref{help14} yields that
 \begin{align*}
    V^{p,t}_n(f)
     & = b^n \EE\left( \left\vert  \sum_{m=1}^n \xi_{n-m}  (\psi(b^{-1}))^{n-m} b^{-m} Y_m \right\vert^p 
                      \bone_{\{ R_n\leq \lfloor tb^n\rfloor\}}  \right) \\
     & = ( (\psi(b^{-1}))^p b )^n \EE\left( \left\vert  \sum_{m=1}^n  \xi_{n-m}(\psi(b^{-1})b)^{-m}  Y_m \right\vert^p 
                        \bone_{\{ R_n\leq \lfloor tb^n\rfloor\}}  \right),
                        \qquad  t\in[0,1], 
 \end{align*}
 as desired.
\end{proof}
 
Lemma \ref{Lem_p_var_expression} shows that the distributions of $(Y_1,\ldots,Y_n)$, $n\in\NN$, play a crucial role 
 in the investigation of $V^{p,1}_n(f)$, $n\in\NN$, where $Y_m$, $m\in\NN$, are defined in \eqref{help33}.
In general, we have little information about these distributions.
In the paragraph after the proof of Theorem \ref{Thm_Riesz}, we recall a special case corresponding to Takagi functions  
 due to Schied and Zhang \cite[Proposition 3.3]{SchZha1},
 when $Y_m$, $m\in\NN$, are independent and identically distributed random variables having Rademacher distribution,
 i.e., $\PP(Y_1=1) = \PP(Y_1=-1) = \frac{1}{2}$.

Note that, in general, $V^{p,1}_n(f)$ depends on the signs $\xi_m$, $m\in\{0,1,\ldots,n-1\}$, see formulae \eqref{help14} and \eqref{help15}.
In case of a signed Takagi-Landsberg function $g$ with Hurst parameter $H\in(0,1)$ given by \eqref{signed_TL}, Mishura and Schied \cite[page 266]{MisSch} pointed out that
 $V^{\frac{1}{H},1}_n(g)$ does not depend on the choices of $\theta_{m,k}\in\{-1,1\}$ in \eqref{signed_TL}.

For each $n\in\NN$, define the random variable
 \begin{align}\label{help_Wn}
     W_n := b^{n(\gamma-1)} \sum_{m=1}^n \xi_{n-m} \psi(b^{m-n}) b^{n-m} Y_m.
 \end{align}

\begin{Thm}\label{Thm3}
Let us consider the function $f$ defined by \eqref{help7}, and suppose that $\psi(b^{-1}) < b^{-\gamma}$,
 where $\gamma\in(0,1]$ is the exponent of H\"older continuity
 for the function $\phi$.
 \begin{itemize}
   \item[(i)] The sequence $(W_n)_{n\in\NN}$ given in \eqref{help_Wn} is uniformly bounded, namely,
              \begin{align*}
                  \vert W_n\vert \leq  \frac{LC}{1 - \psi(b^{-1}) b^\gamma}, \qquad n\in\NN,
               \end{align*}
              holds $\PP$-almost surely, where the constant $C$ is given by \eqref{phi_Holder}.
   \item[(ii)] If $p>\frac{1}{\gamma}$, then $\lim_{n\to\infty}V^{p,t}_n(f)=0$ for all $t\in[0,1]$,
   \item[(iii)] If $p = \frac{1}{\gamma}$, then 
                 \[
                    \limsup_{n\to\infty}V^{p,t}_n(f)
                    \leq \limsup_{n\to\infty}  \EE( \vert W_n\vert ^p \bone_{\{ R_n\leq \lfloor tb^n\rfloor\}})
                    \leq \left(\frac{LC}{1-\psi(b^{-1})b^\gamma}\right)^p t
                  \]  
               for all $t\in[0,1]$, where the constant $C$ is given by \eqref{phi_Holder}.
 \end{itemize}
\end{Thm}

For the case $1\leq p < \frac{1}{\gamma}$ with $\gamma<1$, see part (i) of Theorem \ref{Thm_last_parts}, where we do not need the assumption 
 $\psi(b^{-1}) < b^{-\gamma}$.

\noindent{\bf Proof of Theorem \ref{Thm3}.}
(i): Using that $\psi$ is nonnegative and submultiplicative, by the inequality \eqref{Y_estimate},
 $\PP$-almost surely, for all $n\in\NN$, it holds that
 \begin{align*}
  \vert W_n\vert
   & \leq L b^{n(\gamma - 1)} \sum_{m=1}^n (\psi(b^{-1}))^{n-m} b^{n-m} \vert Y_m\vert
    \leq L C b^{n(\gamma - 1)} \sum_{m=1}^n (\psi(b^{-1}) b)^{n-m} b^{m(1-\gamma)}\\
   & = L C (\psi(b^{-1}) b)^n  b^{n(\gamma - 1)} \sum_{m=1}^n (\psi(b^{-1}) b^\gamma)^{-m}
  \end{align*}  
  \begin{align*}  
   & = L C (\psi(b^{-1}) b^\gamma)^n (\psi(b^{-1}) b^\gamma)^{-1} \frac{(\psi(b^{-1}) b^\gamma)^{-n} - 1}{(\psi(b^{-1}) b^\gamma)^{-1}-1}\\
   & = L C \cdot\frac{ 1- (\psi(b^{-1}) b^\gamma)^n }{ 1- \psi(b^{-1}) b^\gamma}
    \leq \frac{LC}{1 - \psi(b^{-1}) b^\gamma}.
 \end{align*}

(ii) and (iii):
Using \eqref{help14}, for all $n\in\NN$, $p\geq 1$, and $t\in[0,1]$, we have
 \begin{align}\label{help35}
  \begin{split}
  V^{p,t}_n(f) & = b^n \EE\left( \left\vert  \sum_{m=1}^n \xi_{n-m} \psi(b^{m-n}) b^{-m} Y_m \right\vert^p 
                                        \bone_{\{ R_n\leq \lfloor tb^n\rfloor\}} \right)\\
               & = b^{n(1-p)} \EE\left( \left\vert  \sum_{m=1}^n \xi_{n-m} \psi(b^{m-n}) b^{n-m} Y_m \right\vert^p 
                                         \bone_{\{ R_n\leq \lfloor tb^n\rfloor\}} \right)\\
               & = b^{n(1-p)} b^{n(1-\gamma)p} \EE\big( \vert W_n \vert^p  \bone_{\{ R_n\leq \lfloor tb^n\rfloor\}} \big)
                 = b^{n(1-\gamma p)} \EE\big( \vert W_n \vert^p \bone_{\{ R_n\leq \lfloor tb^n\rfloor\}} \big).
 \end{split}
 \end{align}
By part (i),  we have that
 \begin{align*}  
  \EE\big( \vert W_n \vert^p  \bone_{\{ R_n\leq \lfloor tb^n\rfloor\}} \big) 
     \leq \left(\frac{LC}{1-\psi(b^{-1})b^\gamma}\right)^p
          \PP(R_n\leq \lfloor tb^n\rfloor),\qquad n\in\NN,\quad t\in[0,1],
 \end{align*} 
 where 
 \[
   \PP(R_n\leq \lfloor tb^n\rfloor) 
   = 1\wedge \frac{\lfloor tb^n\rfloor+1}{b^n}
   =\begin{cases}
      \frac{\lfloor tb^n\rfloor+1}{b^n} & \text{if $t\in[0,1)$,}\\
      1 & \text{if $t=1$,} 
    \end{cases} 
    \qquad n\in\NN, 
 \]
 since $R_n$ is uniformly distributed on $\{0,1,\ldots,b^n-1\}$, and hence 
 $\PP(R_n\leq \lfloor tb^n\rfloor) \to t$ as $n\to\infty$ for all $t\in[0,1]$.
If $p>\frac{1}{\gamma}$, then $b^{n(1-\gamma p)}\to 0$ as $n\to\infty$.
If $p=\frac{1}{\gamma}$, then $b^{n(1-\gamma p)}=1$ for all $n\in\NN$. 
Therefore, using \eqref{help35}, we get parts (ii) and (iii).
\proofend

 In the next remark, we point out the fact that part (ii) of Theorem \ref{Thm3} is in fact a consequence of part (iii) of Theorem \ref{Thm3}.
In this way, we also give an alternative proof of part (ii) of Theorem \ref{Thm3}.

\begin{Rem}\label{Rem3}
Let us consider the function $f$ defined by \eqref{help7}.
Suppose that $\psi(b^{-1}) < b^{-\gamma}$ and let $p>\frac{1}{\gamma}$.
Similarly as in the proof of part (i) of Lemma \ref{Lem_change_point}, for all $t\in[0,1]$ we get that
 \begin{align*}
  V^{p,t}_n(f) &=\sum_{k=0}^{\lfloor tb^n\rfloor} \vert f((k+1)b^{-n}) - f(kb^{-n})\vert^p\\
  &\leq \left(  \sup_{k\in\{0,1,\ldots,\lfloor tb^n\rfloor\}}  \vert f((k+1)b^{-n}) - f(kb^{-n})\vert  \right)^{p-\frac{1}{\gamma}}
           \sum_{k=0}^{\lfloor tb^n\rfloor} \vert f((k+1)b^{-n}) - f(kb^{-n})\vert^{\frac{1}{\gamma}}\\
  &= \left(  \sup_{k\in\{0,1,\ldots,\lfloor tb^n\rfloor\}}  \vert f((k+1)b^{-n}) - f(kb^{-n})\vert  \right)^{p-\frac{1}{\gamma}}
                 V^{\frac{1}{\gamma},t}_n(f).
 \end{align*}
Using part (iii) of Theorem \ref{Thm3} and that $\limsup_{n\to\infty}(a_nb_n)\leq (\limsup_{n\to\infty}(a_n))(\limsup_{n\to\infty}(b_n))$
 for any sequences $(a_n)_{n\in\NN}$ and $(b_n)_{n\in\NN}$ of nonnegative real numbers, we get that
 \begin{align*}
  \limsup_{n\to\infty} V^{p,t}_n(f)
    &\leq \limsup_{n\to\infty} \left(  \sup_{k\in\{0,1,\ldots,\lfloor tb^n\rfloor\}}  \vert f((k+1)b^{-n}) - f(kb^{-n})\vert  \right)^{p-\frac{1}{\gamma}}
               \limsup_{n\to\infty} V^{\frac{1}{\gamma},t}_n(f) \\
  &\leq \limsup_{n\to\infty} \left(  \sup_{k\in\{0,1,\ldots,\lfloor tb^n\rfloor\}}  \vert f((k+1)b^{-n}) - f(kb^{-n})\vert  \right)^{p-\frac{1}{\gamma}}
         \cdot \left(\frac{LC}{1-\psi(b^{-1})b^\gamma}\right)^{\frac{1}{\gamma} } \\
  &= 0\cdot \left(\frac{LC}{1-\psi(b^{-1})b^\gamma}\right)^{\frac{1}{\gamma} } =0,
  \qquad  t\in[0,1],
 \end{align*}
 since $f$ is uniformly continuous on $[0,1]$.
Hence part (ii) of Theorem \ref{Thm3} holds.
\proofend
\end{Rem}

\begin{Thm}\label{Thm4}
Let us consider the function $f$ defined by \eqref{help7}, and suppose that $\psi(b^{-1}) = b^{-\gamma}$,
 where $\gamma\in(0,1]$ is the exponent of H\"older continuity for the function $\phi$.
 \begin{itemize}
   \item[(i)] If $p>\frac{1}{\gamma}$, then $\lim_{n\to\infty} V^{p,t}_n(f) =0$ for all $t\in[0,1]$.
   \item[(ii)] If $p=\frac{1}{\gamma}$, then
               \begin{align}\label{help44}
                  V^{p,t}_n(f)\leq \left(L\sum_{m=1}^n b^{-m(1-\gamma)} 
                      \left(\EE\big(\vert Y_m\vert^{\frac{1}{\gamma}} \bone_{\{ R_n\leq \lfloor tb^n\rfloor\}} \big)
                               \right)^\gamma \right)^{\frac{1}{\gamma}}
                               \leq (LCn)^{\frac{1}{\gamma}} \left(1\wedge \frac{\lfloor tb^n\rfloor+1}{b^n}\right)
               \end{align}
               for all $t\in[0,1]$ and $n\in\NN$, where $C$ is given by \eqref{phi_Holder}
               and $Y_m$, $m\in\NN$, are defined in \eqref{help33}.
               Consequently, $\limsup_{n\to\infty} n^{-\frac{1}{\gamma}}V^{p,t}_n(f) \leq (LC)^{\frac{1}{\gamma}} t$, $t\in[0,1]$.
  \end{itemize}
\end{Thm}

For the case $1\leq p < \frac{1}{\gamma}$ with $\gamma<1$, see part (i) of Theorem \ref{Thm_last_parts}, where we do not need the assumption 
 $\psi(b^{-1}) = b^{-\gamma}$.

\noindent{\bf Proof of Theorem \ref{Thm4}.}
(i): Suppose that $p>\frac{1}{\gamma}$.
     Since $0\leq V^{p,t}_n(f) \leq V^{p,1}_n(f)$, $t\in[0,1]$, it suffices to prove (i) for $V^{p,1}_n(f)$.
     Using that $b^{-n}\leq \frac{1}{2}$, $n\in\NN$, part (iii) of Proposition \ref{Pro1} with $\eta:=1$ implies that
     there exists a constant $C_1\in\RR_{++}$ such that
     \begin{align*}
              \vert f((k+1)b^{-n}) - f(kb^{-n})\vert
                \leq C_1 b^{-n\gamma} \log_b(b^n)
                = C_1 n b^{-n\gamma}
     \end{align*}
     for each $n\in\NN$ and $k\in\{0,1,\ldots,b^n-1\}$.
     For a possible choice of the constant $C_1$, see \eqref{constant_C1} with $\eta:=1$.
     Hence, using also part (i) of Remark \ref{Rem2}, we get that
     \begin{align}\label{help40}
      \begin{split}
            V^{p,1}_n(f)& = \sum_{k=0}^{b^n-1} \vert f((k+1)b^{-n}) - f(kb^{-n})\vert^p \\
                        & \leq C_1^p \sum_{k=0}^{b^n-1} n^p b^{-n\gamma p}
                          = C_1^p n^p b^{n(1-\gamma p)}
                         \to 0 \qquad \text{as $n\to\infty$.}
     \end{split}
     \end{align}

(ii): Suppose that $p=\frac{1}{\gamma}$.
Using \eqref{help14}, we have for all $n\in\NN$
 \begin{align*}
    (V^{p,t}_n(f))^{\frac{1}{p}}
       &= b^{\frac{n}{p}} \left(\EE\left( \left\vert  \sum_{m=1}^n \xi_{n-m} \psi(b^{m-n}) b^{-m} Y_m \right\vert^p 
                                          \bone_{\{ R_n\leq \lfloor tb^n\rfloor\}}  \right) \right)^{\frac{1}{p}}\\
       &\leq L b^{\frac{n}{p}} \sum_{m=1}^n (\psi(b^{-1}))^{n-m} b^{-m} \Big(\EE(\vert Y_m\vert^p \bone_{\{ R_n\leq \lfloor tb^n\rfloor\}})\Big)^{\frac{1}{p}} \\
       & = L b^{n\gamma} (\psi(b^{-1}))^n \sum_{m=1}^n (\psi(b^{-1}) b)^{-m} \Big(\EE\big(\vert Y_m\vert^{\frac{1}{\gamma}}
             \bone_{\{ R_n\leq \lfloor tb^n\rfloor\}}  \big)\Big)^{\gamma}\\
       & = L\sum_{m=1}^n b^{-m(1-\gamma)} \Big(\EE\big(\vert Y_m\vert^{\frac{1}{\gamma}} 
              \bone_{\{ R_n\leq \lfloor tb^n\rfloor\}}  \big)\Big)^{\gamma},
 \end{align*}
 where the inequality follows by an application of Minkowski's inequality together with
 \[
  \vert \psi(b^{m-n}) \vert = \psi(b^{m-n}) \leq (\psi(b^{-1}))^{n-m}, \qquad m\in\{1,\ldots,n\}, \;\; n\in\NN,
 \]
 due to the non-negativity and submultiplicativity of $\psi$.
This implies the first inequality in \eqref{help44}.
Using \eqref{Y_estimate}, we get
  \[
     \sum_{m=1}^n b^{-m(1-\gamma)} \Big(\EE\big(\vert Y_m\vert^{\frac{1}{\gamma}}  
                \bone_{\{ R_n\leq \lfloor tb^n\rfloor\}} \big)\Big)^{\gamma} 
       \leq C n (\PP(R_n\leq \lfloor tb^n\rfloor))^\gamma, \qquad n\in\NN, 
        \quad t\in[0,1],
 \]
 where the constant $C$ is given by \eqref{phi_Holder}.
Since $R_n$ is uniformly distributed on $\{0,1,\ldots,b^{n}-1\}$, we get
  $\PP(R_n\leq \lfloor tb^n\rfloor) = 1\wedge \frac{\lfloor tb^n\rfloor+1}{b^n}$, $n\in\NN$, $t\in[0,1]$,
 yielding the second inequality in \eqref{help44}.
\proofend

For each $n\in\NN$, define the random variable
 \begin{align}\label{help_Tn}
     T_n := (\psi(b^{-1}))^{-n} \sum_{m=1}^n \xi_{n-m} \psi(b^{m-n}) b^{-m} Y_m.
 \end{align}
Note that if $\psi(b^{-1})=b^{-\gamma}$, then $T_n=W_n$, $n\in\NN$, and $r=\gamma$, where $W_n$, $n\in\NN$, is given in \eqref{help_Wn},
 and $r$ is given in \eqref{beta_def}.

\begin{Thm}\label{Thm5}
Let us consider the function $f$ defined by \eqref{help7}, and suppose that $\psi(b^{-1}) > b^{-\gamma}$.
Recall that $\gamma\in(0,1]$ is the exponent of H\"older continuity
 for the function $\phi$, and $r=-\log_b(\psi(b^{-1}))\in(0,\gamma)$ is given in \eqref{beta_def}.
 \begin{itemize}
   \item[(i)] The sequence $(T_n)_{n\in\NN}$ given in \eqref{help_Tn} is uniformly bounded, namely,
             \begin{align*}
               \vert T_n\vert \leq  \frac{LC}{\psi(b^{-1}) b^\gamma - 1}, \qquad n\in\NN,
               \end{align*}
               holds $\PP$-almost surely, where the constant $C$ is given by \eqref{phi_Holder}.
   \item[(ii)] If $p>\frac{1}{r}$, then $\lim_{n\to\infty} V^{p,t}_n(f) =0$ for all $t\in[0,1]$.
   \item[(iii)] If $p= \frac{1}{r}$, then
               \begin{align*}
                  \limsup_{n\to\infty}  V^{p,t}_n(f) \leq \limsup_{n\to\infty} \EE( \vert T_n\vert^p \bone_{\{ R_n\leq \lfloor tb^n\rfloor\}} )
                      \leq \left( \frac{LC}{\psi(b^{-1}) b^\gamma - 1} \right)^p t
               \end{align*}
               for all $t\in[0,1]$, where the constant $C$ is given by \eqref{phi_Holder}.
  \end{itemize}
\end{Thm}

For the case $1\leq p < \frac{1}{r}$, see part (ii) of Theorem \ref{Thm_last_parts}, where we do not need the assumption 
 $\psi(b^{-1}) > b^{-\gamma}$, and instead of $r\in(0,\gamma)$ we allow $r\in(0,1)$, which is equivalent to 
 $\psi(b^{-1}) > b^{-1}$.

\noindent{\bf Proof of Theorem \ref{Thm5}.}
The fact that $r\in(0,\gamma)$ readily follows from the assumption that $b^{-\gamma}< \psi(b^{-1})<1$.

(i):
Using that $\psi$ is nonnegative and submultiplicative, by the inequality \eqref{Y_estimate},
  $\PP$-almost surely, for all $n\in\NN$, it holds that
 \begin{align*}
  \vert T_n\vert
   & \leq L (\psi(b^{-1}))^{-n} \sum_{m=1}^n (\psi(b^{-1}))^{n-m} b^{-m} \vert Y_m\vert
    \leq L C \sum_{m=1}^n (\psi(b^{-1}) b^\gamma)^{-m}\\
   & \leq L C \left(\frac{1}{1 - (\psi(b^{-1}) b^\gamma)^{-1} } - 1 \right)
     = \frac{LC}{\psi(b^{-1}) b^\gamma - 1}.
 \end{align*}

(ii) and (iii):
Using \eqref{help14}, we have
 \begin{align}\label{help36}
  \begin{split}
  V^{p,t}_n(f) & = b^n \EE\left( \left\vert  \sum_{m=1}^n \xi_{n-m} \psi(b^{m-n}) b^{-m} Y_m \right\vert^p 
                                         \bone_{\{ R_n\leq \lfloor tb^n\rfloor\}} \right)\\
               & = b^n \EE\big( (\psi(b^{-1}))^{np} \vert T_n\vert^p  \bone_{\{ R_n\leq \lfloor tb^n\rfloor\}} \big) \\
               & = (\psi(b^{-1}) b^{\frac{1}{p}})^{np}\EE\big(\vert T_n\vert^p \bone_{\{ R_n\leq \lfloor tb^n\rfloor\}}  \big)
                = b^{n(1-pr)}\EE\big(\vert T_n\vert^p \bone_{\{ R_n\leq \lfloor tb^n\rfloor\}} \big)
  \end{split}
 \end{align}
 for all $n\in\NN$, $p\geq 1$ and $t\in[0,1]$, where we used that 
  \[
  (\psi(b^{-1}) b^{\frac{1}{p}})^{np} = (b^{-r} b^{\frac{1}{p}})^{np} = b^{n(1-pr)}.
  \]
By part (i), we have that
 \begin{align*}
  \EE\big(\vert T_n\vert^p \bone_{\{ R_n\leq \lfloor tb^n\rfloor\}} \big)
     \leq \left( \frac{LC}{\psi(b^{-1}) b^\gamma - 1} \right)^p
           \PP(R_n\leq \lfloor tb^n\rfloor), \qquad n\in\NN, \quad t\in[0,1],
 \end{align*}
 where $\PP(R_n\leq \lfloor tb^n\rfloor) = 1\wedge \frac{\lfloor tb^n\rfloor+1}{b^n}$, $n\in\NN$, $t\in[0,1]$, 
 since $R_n$ is uniformly distributed on $\{0,1,\ldots,b^n-1\}$,
 and hence $\PP(R_n\leq \lfloor tb^n\rfloor) \to t$ as $n\to\infty$ for all $t\in[0,1]$.
If $p>\frac{1}{r}$, we have $b^{n(1-pr)}\to 0$ as $n\to\infty$.
If $p=\frac{1}{r}$, then $b^{n(1-pr)} = 1$ for all $n\in\NN$.
Therefore, using \eqref{help36},  we get parts (ii) and (iii).
\proofend

\begin{Rem}
(i) Similarly as in Remark \ref{Rem3}, one can check that part (ii) of Theorem \ref{Thm5} is in fact a consequence of part (iii) of Theorem \ref{Thm5},
 and it is an alternative proof for part (ii) of Theorem \ref{Thm5}.

(ii) If $\xi_m=1$, $m\in\ZZ_+$, and $\psi$ is multiplicative 
 in the definition \eqref{help7} of $f$, then part (ii) of Theorem \ref{Thm5} follows from Proposition A.2 in Schied and Zhang \cite{SchZha2}
 taking into account also that the existence of a continuous $\big(\frac{1}{r}\big)^{\mathrm{th}}$-variation
 function of $f$ along the sequence of $b$-adic partitions implies that for any $p>\frac{1}{r}$, the continuous
 $p^{\mathrm{th}}$-variation function of $f$ along the sequence of $b$-adic partitions exists as well
 and it is identically $0$ (see part (i) of Lemma \ref{Lem_change_point}).
\proofend
\end{Rem}

\begin{Thm}\label{Thm_last_parts}
Let us consider the function $f$ defined by \eqref{help7}.
Recall that $\gamma\in(0,1]$ is the exponent of H\"older continuity for the function $\phi$,
 and $r=-\log_b(\psi(b^{-1}))$ is given in \eqref{beta_def}.
\begin{itemize}
  \item[(i)] If $\gamma\in(0,1)$, $p\in[1,\frac{1}{\gamma})$ 
             and $\PP(\liminf_{n\to\infty} \vert W_n\vert >0)>0$, then
             \begin{align}\label{help22}
                \lim_{n\to\infty} V^{p,1}_n(f) = \infty,
             \end{align}
             where $W_n$, $n\in\NN$, are given in \eqref{help_Wn}.
             In particular, in this case, by choosing $p=1$, we have that $f$ is not of bounded variation.
  \item[(ii)] If $r\in(0,1)$ (i.e., $\psi(b^{-1}) > b^{-1}$), $p\in[1,\frac{1}{r})$ and $\PP(\liminf_{n\to\infty} \vert T_n\vert > 0)>0$, then
               \begin{align}\label{help25}
                  \lim_{n\to\infty} V^{p,1}_n(f) = \infty,
               \end{align}
               where $T_n$, $n\in\NN$, are given in \eqref{help_Tn}.
               In particular, by choosing $p=1$, we have that $f$ is not of bounded variation.
\end{itemize}
\end{Thm}

\noindent{\bf Proof.}
(i). If $\gamma\in(0,1)$ and $p\in[1,\frac{1}{\gamma})$, then $b^{n(1-\gamma p)}\to\infty$ as $n\to\infty$.
Further, since $\PP(\liminf_{n\to\infty} \vert W_n\vert >0)>0$,
 by Fatou's lemma, we have that
 \[
     \liminf_{m\to\infty} \EE\big( \vert W_m \vert^p \big) \geq \EE\big( \liminf_{m\to\infty}  \vert W_m \vert^p \big) >0.
 \]
Hence, using \eqref{help35} (where we did not use the assumption that $\psi(b^{-1}) < b^{-\gamma}$), we get \eqref{help22}.
Indeed, for sufficiently large $n\in\NN$, we have that
 \[
  \EE\big( \vert W_n \vert^p \big) \geq \frac{1}{2}\EE\big( \liminf_{m\to\infty}  \vert W_m \vert^p \big) >0,
 \]
 yielding that
 \[
 V^{p,1}_n(f) \geq \frac{1}{2} b^{n(1-\gamma p)} \EE\big( \liminf_{m\to\infty}  \vert W_m \vert^p \big)
 \]
 for sufficiently large $n\in\NN$, and the right hand side of the previous inequality tends to $\infty$ as $n\to\infty$.

(ii). If $r\in(0,1)$ (i.e., $\psi(b^{-1}) > b^{-1}$) and $p\in[1,\frac{1}{r})$, then $(\psi(b^{-1}) b^{\frac{1}{p}})^{np} \to\infty$ as $n\to\infty$,
 since $\psi(b^{-1}) b^{\frac{1}{p}} > \psi(b^{-1}) b^r = 1$.
Further, since $\PP(\liminf_{n\to\infty} \vert T_n\vert >0)>0$, by Fatou's lemma,
 we have that
 \[
   \liminf_{m\to\infty}  \EE\big( \vert T_m \vert^p \big) \geq \EE\big( \liminf_{m\to\infty} \vert T_m\vert^p \big) > 0.
 \]
Hence, using \eqref{help36} (where we did not use the assumption that $\psi(b^{-1}) > b^{-\gamma}$), we get \eqref{help25}.
Indeed, for sufficiently large $n\in\NN$, we have that
 \[
    \EE\big( \vert T_n \vert^p \big) \geq \frac{1}{2} \EE\big( \liminf_{m\to\infty} \vert T_m\vert^p \big)>0,
 \]
 yielding that
 \[
   V^{p,1}_n(f) \geq \frac{1}{2}(\psi(b^{-1}) b^{\frac{1}{p}})^{np} \EE\big( \liminf_{m\to\infty} \vert T_m\vert^p \big)
 \]
 for sufficiently large $n\in\NN$, and the right-hand side of this inequality tends to $\infty$ as $n\to\infty$.
\proofend

In the case $\psi(b^{-1})=b^{-\gamma}$, we have $T_n=W_n$, $n\in\NN$, and hence, 
 in this case, in parts (i) and (ii) of Theorem \ref{Thm_last_parts}, 
 the conditions $\PP(\liminf_{n\to\infty} \vert W_n\vert >0)>0$ and 
 $\PP(\liminf_{n\to\infty} \vert T_n\vert >0)>0$ coincide.

Concerning part (i) of Theorem \ref{Thm_last_parts}, we note that 
  we could not find a well-useable sufficient condition under which $\PP(\liminf_{n\to\infty} \vert W_n\vert >0)\!>\!0$ holds,
 even in case of a multiplicative $\psi$. 
However, next we give an example when this condition holds.
 
\begin{Ex}\label{Ex_1}
Let $b\in\NN\setminus\{1\}$, $\gamma\in(0,1]$, $\xi_m:=1$, $m\in\ZZ_+$, 
 $\phi:\RR\to\RR$, $\phi(t):=(\min_{z\in\ZZ} \vert t-z\vert)^\gamma$, $t\in\RR$,
 and $\psi:\RR_{++}\to\RR_{++}$, 
 \begin{equation}\label{psi_Ex_1}
    \psi(x):=x^\alpha\left(1+\frac{1}{1+x}\right)^{\beta}=:x^\alpha A(x), \qquad x\in\RR_{++}, 
 \end{equation}
 where $0<\beta<\alpha<\gamma$.
Since $\RR\ni t\mapsto \min_{z\in\ZZ} \vert t-z\vert$ is Lipschitz continuous and $\gamma\in(0,1]$,
 it is known that $\phi$ is H\"older continuous with exponent $\min(\gamma,1)=\gamma$.
Further, $1<A(x)<2^{\beta}<b^{\alpha}$, $x\in\RR_{++}$, and $\psi$ is submultiplicative, but not multiplicative.
Indeed, since the function $\RR_{++}\ni x\mapsto x^\alpha$ is multiplicative, it is enough to check that 
 $A$ is submultiplicative, but not multiplicative.
For all $x,y\in\RR_{++}$, we have that the inequality $A(xy)< A(x)A(y)$ holds if and only if
 \[
   1+\frac{1}{1+xy} < \left( 1+\frac{1}{1+x} \right)\left( 1+\frac{1}{1+y} \right),
 \]
 which is equivalent to
 \[
 \frac{1}{1+x} + \frac{1}{1+y} + \frac{1}{(1+x)(1+y)} - \frac{1}{1+xy} >0.
 \]
After some algebraic manipulation, we have that this is equivalent to any of the  following two inequalities
 \[
   (1+y)(1+xy) + (1+x)(1+xy) + 1+xy - (1+x)(1+y) >0,
  \] 
 and 
 \[
  x(y^2-1) + (1+xy)(2+x) > 0.
 \]
Since $x(y^2-1) + (1+xy)(2+x) \geq xy(x+y)>0$ for all $x,y\in\RR_{++}$, 
 we have that $A$ is submultiplicative, but not multiplicative, as desired.
Since $A(x)\in(1,b^\alpha)$, $x\in\RR_{++}$, the condition $\psi(b^{-1})\in(0,1)$ in the definition \eqref{help7} of the function $f$ holds as well.
Note also that $\psi(b^{-1})=b^{-\alpha}A(b^{-1}) > b^{-\gamma}>b^{-1}$ following from the inequalities $A(x)>1$, $x\in\RR_{++}$, and $\alpha<\gamma$.
Moreover, for the random variables $W_n$, $n\in\NN$, defined in \eqref{help_Wn}, we get that
 \begin{align}\label{help51}
  W_n = b^{n(\gamma-1)} \sum_{m=1}^n b^{\alpha(m-n)} \left(1+\frac{1}{1+b^{m-n}}\right)^{\beta}  b^{n-m} Y_m
       = b^{n(\gamma-\alpha)}\sum_{m=1}^n b^{-m(1-\alpha)} A(b^{m-n} ) Y_m
 \end{align}
 for each $n\in\NN$.
Further, for each $N\in\NN$ and $\omega\in\{ U_1=0,U_2=0,\ldots,U_N=0\}$, we have that $R_m(\omega)=0$, $m\in\{1,\ldots,N\}$, and 
 \begin{align*}
  Y_m(\omega) = \lambda_{m,R_m(\omega)} = \lambda_{m,0} 
              = b^m (\phi(b^{-m}) - \phi(0))
              = b^m \phi(b^{-m})
              = b^{m(1-\gamma)},
              \qquad m\in\{1,\ldots,N\}.
 \end{align*}
Consequently, for each $N\in\NN$ and $\omega\in\{ U_1=0,U_2=0,\ldots,U_N=0\}$, we have that 
 \begin{align*}
  \sum_{m=1}^N b^{-m(1-\alpha)} A(b^{m-n} ) Y_m(\omega)
   = \sum_{m=1}^N b^{-m(\gamma-\alpha)} A(b^{m-n} )
   \geq \sum_{m=1}^N b^{-m(\gamma-\alpha)}
   \geq  b^{-(\gamma-\alpha)} >0.
 \end{align*}
Further, by \eqref{Y_estimate}, we have that 
 \[
   b^{-m(1-\alpha)} A(b^{m-n} ) \vert Y_m\vert\leq C b^{-m(1-\alpha)} b^{m(1-\gamma)} A(b^{m-n} ) 
      = C b^{-m(\gamma-\alpha)} A(b^{m-n})
      \leq C b^\alpha b^{-m(\gamma-\alpha)}
 \]
 for each $m\in\NN$, where $C\in\RR_{++}$ is given by \eqref{phi_Holder}.   
Therefore, for each $N\in\NN$ and $n\geq N+1$, $n\in\NN$, we get that 
 \begin{align*}
\left\vert  \sum_{m=N+1}^n b^{-m(1-\alpha)} A(b^{m-n} ) Y_m \right\vert 
   \leq C b^\alpha \sum_{m=N+1}^n  b^{-m(\gamma-\alpha)} 
   \leq  Cb^\alpha  \sum_{m=N+1}^\infty  b^{-m(\gamma-\alpha)} 
   = Cb^\alpha\frac{b^{-(N+1)(\gamma-\alpha)}}{1-b^{-(\gamma-\alpha)}},
 \end{align*}
 where the right-hand side tends to $0$ as $N\to\infty$.
Consequently, we can choose an $N_0\in\NN$ such that
  \[
     Cb^\alpha\frac{b^{-(N_0+1)(\gamma-\alpha)}}{1-b^{-(\gamma-\alpha)}} < \frac{1}{2}b^{-(\gamma-\alpha)}.    
 \] 
Hence, by the reverse triangle inequality, if $\omega\in\{ U_1=0,U_2=0,\ldots,U_{N_0}=0\}$, then 
 \begin{align*}
  \left\vert  \sum_{m=1}^n b^{-m(1-\alpha)} A(b^{m-n} ) Y_m(\omega)\right\vert 
    & \geq \left\vert  \sum_{m=1}^{N_0} b^{-m(1-\alpha)} A(b^{m-n} ) Y_m(\omega) \right\vert 
           - \left\vert  \sum_{m=N_0+1}^n b^{-m(1-\alpha)} A(b^{m-n} ) Y_m(\omega) \right\vert \\
    & \geq b^{-(\gamma-\alpha)} - \frac{1}{2} b^{-(\gamma-\alpha)} = \frac{1}{2}b^{-(\gamma-\alpha)}>0
 \end{align*}
  for each $n\geq N_0+1$, $n\in\NN$.
Therefore, if $\omega\in\{ U_1=0,U_2=0,\ldots,U_{N_0}=0\}$, then, using also \eqref{help51}, we have 
 \begin{align*}
  \vert W_n(\omega)\vert
   \geq b^{n(\gamma-\alpha)} \frac{1}{2}b^{-(\gamma-\alpha)}
         = \frac{1}{2}b^{(n-1)(\gamma-\alpha)},
      \qquad n\geq N_0+1, \;\; n\in\NN,
 \end{align*}
 yielding that  $\vert W_n(\omega)\vert\to\infty$ as $n\to\infty$.
Hence 
\[
 \{ U_1=0,U_2=0,\ldots,U_{N_0}=0\}\subseteq \Big\{ \liminf_{n\to\infty}\vert W_n \vert > 0\Big\}.
 \]
Since $U_1,\ldots,U_{N_0}$ are i.i.d.\ such that $U_1$ is uniformly distributed on the set $\{0,1,\ldots,b-1\}$, we get
 \[
   \PP\Big(\liminf_{n\to\infty}\vert W_n\vert > 0 \Big) \geq \PP(U_1=0,\ldots,U_{N_0}=0) = b^{-N_0} > 0,
 \]
 yielding that $\PP(\liminf_{n\to\infty}\vert W_n \vert > 0) > 0$, as desired.
Consequently, by part (i) of Theorem \ref{Thm_last_parts}, if $\gamma\in(0,1)$, then it holds that 
 \begin{equation}\label{V_infty_Ex_1}
     \lim_{n\to\infty} V^{p,1}_n(f) = \infty \qquad \text{for all $p\in\left[1,\frac{1}{\gamma}\right)$.}
 \end{equation}
 
Furthermore, the random variables $T_n$, $n\in\NN$, defined in \eqref{help_Tn}, take the form
 \[
   T_n = (A(b^{-1}))^{-n} \sum_{m=1}^n b^{-m(1-\alpha)}A(b^{m-n}) Y_m,\qquad n\in\NN.
 \]
We mention that we could not prove that $\PP(\liminf_{n\to\infty} \vert T_n\vert > 0)>0$, 
  and because of it, we could not apply part (ii) of Theorem \ref{Thm_last_parts}.
However, by part (ii) of Theorem \ref{Thm5}, we obtain that
 \[
   \lim_{n\to\infty} V^{p,t}_n(f) = 0 \qquad \text{for all $p\in\left(\frac{1}{r},\infty\right)$ and $t\in[0,1]$,}
 \]
 where 
 \[
  r=-\log_b(\psi(b^{-1})) = \alpha - \beta\log_b\left(1+\frac{b}{b+1}\right)\in(0,\alpha).
  \]
 Note that $r\uparrow\gamma$ as $\alpha\uparrow\gamma$ and $\beta\downarrow0$. 
Hence, for every $p>\frac1{\gamma}$, we can find $0<\beta<\alpha<\gamma$ such that $p>\frac1r$. 
This shows that, for every $p>\frac1{\gamma}$, we can find a submultiplicative function $\psi$ of the form \eqref{psi_Ex_1},
 with an appropriate choice of $0<\beta<\alpha<\gamma$, such that for $\phi(t)=(\min_{z\in\ZZ} \vert t-z\vert)^\gamma$, $t\in\RR$, $\xi_m=1$, $m\in\ZZ_+$,
 and $f$ defined as in \eqref{help7} it holds that $\lim_{n\to\infty} V^{p,t}_n(f) = 0$, $t\in[0,1]$, whereas, for $p<\frac1{\gamma}$ and the same function $f$, we have $\lim_{n\to\infty} V^{p,1}_n(f) = \infty$ by \eqref{V_infty_Ex_1}.
\proofend
\end{Ex}

In what follows, we present some sufficient conditions under which 
 $\PP(\liminf_{n\to\infty} \vert T_n\vert > 0)>0$ holds (appearing as an assumption in part (ii) 
 of Theorem \ref{Thm_last_parts}).

\begin{Pro}\label{Pro_T_n_liminf_sufficient}
Let us consider the function $f$ defined by \eqref{help7} such that $\xi_m=1$, $m\in\ZZ_+$.
Suppose that
 \begin{itemize}
   \item[(i)] $\psi(b^{-1}) > b^{-\gamma}$, 
   \item[(ii)] $\liminf_{n\to\infty}(\psi(b^{-1}))^{-n}\psi(b^{-n})$ belongs to $(0,\infty]$,
   \item[(iii)] $\{0\} \ne \big\{ \phi(b^{-k}) : k\in\NN \big\} \subseteq \RR_+$.
 \end{itemize}
Then we have that
  $\PP(\liminf_{n\to\infty} \vert T_n\vert > 0)>0$.
\end{Pro}

\noindent
{\bf Proof.}
Recall that, by \eqref{Y_estimate}, we have $\vert Y_m\vert\leq C b^{m(1-\gamma)}$, $m\in\NN$, where $C\in\RR_{++}$ is given by \eqref{phi_Holder}.
By the assumption (iii), there exists $M\in\NN$ such that $\phi(b^{-M})>0$.
Recall also that $(U_n)_{n\in\NN}$ is a sequence of independent and identically distributed random variables such that
 $U_1$ is uniformly distributed on the finite set $\{0,1,\ldots,b-1\}$.
 
Choose $N\in\NN$ with $N>M$ and 
 \begin{align}\label{delta_choice}
 \delta\in\left(0, \frac{\phi(b^{-M})}{(\psi(b^{-1}))^M} \liminf_{n\to\infty}(\psi(b^{-1}))^{-n}\psi(b^{-n}) \right)   
 \end{align}
 such that 
 \[
  C \sum_{m=N+1}^\infty (\psi(b^{-1})b^\gamma)^{-m} < \frac{\phi(b^{-M})}{(\psi(b^{-1}))^M} \liminf_{n\to\infty}(\psi(b^{-1}))^{-n}\psi(b^{-n}) - \delta.
 \]
By the assumption (ii), there exists a $\delta$ satisfying \eqref{delta_choice}.
For any $\delta$ satisfying \eqref{delta_choice}, such an $N$ exists, 
 since $\sum_{m=0}^\infty (\psi(b^{-1})b^\gamma)^{-m}$ is convergent due to $\psi(b^{-1})b^\gamma>1$ (following from assumption (i)).
 
If $\omega\in\{ U_1=0,U_2=0,\ldots,U_N=0\}$, then, by \eqref{help32} and \eqref{help33}, for all $m\in\{1,\ldots,N\}$, we get
 \begin{align}\label{help37_sub}
   Y_m(\omega) = \lambda_{m,R_m(\omega)} = \lambda_{m,0}
               = b^m (\phi(b^{-m}) - \phi(0))
               = b^m \phi(b^{-m})
               \geq 0,
 \end{align}
 where in the last inequality we used the assumption that  $\big\{ \phi(b^{-k}) : k\in\NN \big\} \subseteq \RR_+$.
Hence if $\omega\in\{ U_1=0,U_2=0,\ldots,U_N=0\}$, then, by the reverse triangle inequality, \eqref{Y_estimate} and \eqref{help37_sub}, 
 for all $n\geq N$, we have
 \begin{align*}
   \vert T_n(\omega)\vert
        &\geq  (\psi(b^{-1}))^{-n}\left\vert \sum_{m=1}^N \psi(b^{m-n}) b^{-m} Y_m(\omega) \right\vert
               -  (\psi(b^{-1}))^{-n} \left\vert \sum_{m=N+1}^n \psi(b^{m-n}) b^{-m} Y_m(\omega) \right\vert\\
        &\geq (\psi(b^{-1}))^{-n} \psi(b^{M-n}) b^{-M} Y_M(\omega) 
               - (\psi(b^{-1}))^{-n} \left\vert \sum_{m=N+1}^n \psi(b^{m-n}) b^{-m} Y_m(\omega) \right\vert\\
        &\geq (\psi(b^{-1}))^{-n} \psi(b^{M-n}) \phi(b^{-M})
               - (\psi(b^{-1}))^{-n}  \sum_{m=N+1}^n \psi(b^{m-n}) b^{-m} \vert Y_m(\omega)\vert\\
        &\geq (\psi(b^{-1}))^{-n} \psi(b^{M-n}) \phi(b^{-M})  
               - C (\psi(b^{-1}))^{-n}  \sum_{m=N+1}^n \psi(b^{m-n}) b^{-m} b^{m(1-\gamma)}\\
        &\geq (\psi(b^{-1}))^{-n} \psi(b^{M-n}) \phi(b^{-M})  
               - C (\psi(b^{-1}))^{-n} \sum_{m=N+1}^n (\psi(b^{-1}))^{n-m} b^{-m\gamma}\\
        &=(\psi(b^{-1}))^{-n} \psi(b^{M-n}) \phi(b^{-M})  
               -  C \sum_{m=N+1}^n (\psi(b^{-1})b^\gamma)^{-m}\\
        &\geq \frac{\phi(b^{-M})}{(\psi(b^{-1}))^M} \cdot (\psi(b^{-1}))^{M-n} \psi(b^{M-n})  
               -  C \sum_{m=N+1}^\infty (\psi(b^{-1})b^\gamma)^{-m}.
 \end{align*}
Hence, if $\omega\in\{ U_1=0,U_2=0,\ldots,U_N=0\}$, then we have that
 \begin{align*}
   \liminf_{n\to\infty}\vert T_n(\omega)\vert
      \geq \frac{\phi(b^{-M})}{(\psi(b^{-1}))^M} \liminf_{n\to\infty}(\psi(b^{-1}))^{-n}\psi(b^{-n})
           -  C \sum_{m=N+1}^\infty (\psi(b^{-1})b^\gamma)^{-m} >\delta>0,
 \end{align*}   
 that is,
 \[
    \{ U_1=0,\ldots,U_N=0\} \subseteq \Big\{ \liminf_{n\to\infty}\vert T_n \vert > 0\Big\}.
 \]
Since $U_1,\ldots,U_N$ are i.i.d.\ such that $U_1$ is uniformly distributed on the set $\{0,1,\ldots,b-1\}$, we get
 \[
   \PP\Big(\liminf_{n\to\infty}\vert T_n\vert > 0 \Big) \geq \PP(U_1=0,\ldots,U_N=0) = b^{-N} > 0,
 \]
 yielding that $\PP(\liminf_{n\to\infty}\vert T_n \vert > 0) > 0$, as desired.
\proofend

\begin{Rem}
(i) 
If $\psi$ is multiplicative, then the assumption that $\liminf_{n\to\infty} (\psi(b^{-1}))^{-n}\psi(b^{-n})$ belongs to $(0,\infty]$ 
 (appearing in Proposition \ref{Pro_T_n_liminf_sufficient}) holds automatically, 
 since $(\psi(b^{-1}))^{-n}\psi(b^{-n}) = (\psi(b^{-1}))^{-n}(\psi(b^{-1}))^n = 1$, $n\in\NN$.
This is in accordance with the conditions of Proposition \ref{Pro2}, and Remark \ref{Rem_T_n_cond_sufficient} as well.
Note that, in Example \ref{Ex_1}, the assumption (ii) of Proposition \ref{Pro_T_n_liminf_sufficient}
 does not hold, since $(\psi(b^{-1}))^{-n}\psi(b^{-n}) 
  = (1+b/(b+1))^{-n\beta} (1+b^n/(b^n+1))^{\beta} \to 0$ as $n\to\infty$.

(ii) We give some examples for functions $\psi$ for which the assumptions 
 of Proposition \ref{Pro_T_n_liminf_sufficient} are satisfied:
 \begin{itemize}
   \item[(1)] let $\psi:\RR_{++}\to\RR_{++}$, $\psi(x):= x^\alpha$, $x\in\RR_{++}$, where $\alpha\in(0,\gamma)$. 
         Then $\psi$ is multiplicative, $\psi(b^{-1})\in(0,1)$, $\psi(b^{-1}) > b^{-\gamma}$ and 
         $\liminf_{n\to\infty} (\psi(b^{-1}))^{-n}\psi(b^{-n}) = \lim_{n\to\infty} (\psi(b^{-1}))^{-n}\psi(b^{-n}) = \lim_{n\to\infty} 1=1$.
   \item[(2)] let $0<\alpha < \gamma \leq 1 < \beta$, and $\psi:\RR_{++}\to\RR_{++}$ be given by
        \[
           \psi(x):=\begin{cases}
                      x^\alpha & \text{if $x\in(0,1)$,}\\
                      x^\beta & \text{if $x\in[1,\infty)$.}
                    \end{cases}
        \] 
        One can check that the function $\psi$ is submultiplicative (see also Matkowski \cite[Example 1]{Mat}),
        $\psi(b^{-1}) = b^{-\alpha}\in(0,1)$, $\psi(b^{-1}) > b^{-\gamma}$ (due to $\alpha < \gamma$),
        and
        \begin{align*}
          \liminf_{n\to\infty} (\psi(b^{-1}))^{-n}\psi(b^{-n}) = 
           \lim_{n\to\infty} (\psi(b^{-1}))^{-n}\psi(b^{-n})
             = \lim_{n\to\infty} (b^{-\alpha})^{-n} b^{-n\alpha} 
             = \lim_{n\to\infty}  1
             = 1.
        \end{align*}  
  \end{itemize}  
\proofend
\end{Rem}

\section{The special case $\psi(b^{-1}) > b^{-\gamma}$ with a multiplicative function $\psi$}
\label{Sec_Sch_Zha}

In this section, we specialize our results in Section \ref{Sec_Wei_type_frac_func_new} for the case
 $\psi(b^{-1}) > b^{-\gamma}$ with a multiplicative function $\psi$.
This case was considered in Schied and Zhang \cite[part (iii) of Theorem 2.1]{SchZha1} and \cite[Proposition A.2]{SchZha2},
 and it turns out that we can improve their results in question,
 see Propositions \ref{Pro7} and \ref{Pro_Z} and Corollaries \ref{Cor_bounded_var} and \ref{Cor_Z}.
Throughout this section, we assume that $\psi(b^{-1}) > b^{-\gamma}$, and
 we use the terminology that  ``{\sl we always choose the sign $+1$}" for the choice of $\xi_m=+1$ for all
 $m\in\ZZ_+$.
Likewise, we say that ``{\sl we choose alternating signs $+1$ and $-1$}" for the choice of $\xi_m=(-1)^m$ for all
 $m\in\ZZ_+$.

For each $n\in\NN$, let us introduce the random variable
 \begin{align}\label{help34}
   Z_n := \sum_{m=1}^n \big(\pm \psi(b^{-1}) b\big)^{-m} Y_m
       = \sum_{m=1}^n \big(\pm \psi(b^{-1}) \big)^{-m} \Big( \phi((R_m+1)b^{-m}) - \phi(R_mb^{-m}) \Big),
 \end{align}
 where the sign $\pm$ is meant in a way that we always choose the sign $+1$
 or we choose alternating signs $+1$ and $-1$.
For an explanation for these particular choices of signs, see Remark \ref{Rem_Z_choice}.
The definition of $Z_n$, $n\in\NN$, is motivated by the representation of $V_n^{p,1}(f)$ given in \eqref{help15} with $t=1$
 in the case when $\psi$ is multiplicative.
Note also that if $\psi$ is multiplicative, $n\in\NN$, $\xi_m=1$ for all $m=0,\ldots,n-1$,
 and we always choose the sign $+1$ for $Z_n$, then $Z_n=T_n$, where $T_n$ is given in \eqref{help_Tn}.
Similarly, if $\psi$ is multiplicative, $n\in\NN$, $\xi_{m}=(-1)^m$ for all $m=0,\ldots,n-1$,
 and we always choose alternating signs $+1$ and $-1$ for $Z_n$, then $Z_n=(-1)^n T_n$.

\begin{Lem}\label{Lem6}
Let us consider the function $f$ defined by \eqref{help7} such that we always choose the sign $+1$
 or we choose alternating signs $+1$ and $-1$.
Suppose that $\psi(b^{-1}) > b^{-\gamma}$.
Then, for all $p>0$, the family $\{ \vert Z_n\vert^p: n\in\NN\}$ is uniformly integrable and
 \[
 \EE(\vert Z_n\vert^p) \to \EE(\vert Z\vert^p)\qquad \text{as $n\to\infty$,}
 \]
 where $Z:=\sum_{m=1}^\infty  \big(\pm \psi(b^{-1}) b\big)^{-m} Y_m$,
 and the sign $\pm$ for $Z_n$ in \eqref{help34} and for $Z$
 is chosen in the same way as it is chosen for $f$.
\end{Lem}

\begin{proof}
Let $p>0$ be arbitrarily fixed.
Using \eqref{Y_estimate} and that $\vert \xi_m\vert=1$, $m\in\ZZ_+$, we have that
 $\PP$-almost surely for all $n\in\NN$ the inequality holds
 \begin{align*}
  \vert Z_n\vert
  &\leq  \sum_{m=1}^n (\psi(b^{-1})b)^{-m} \vert Y_m\vert\leq
   \sum_{m=1}^n (\psi(b^{-1})b)^{-m} C b^{m(1-\gamma)}  \\ 
  & = C  \sum_{m=1}^n  (\psi(b^{-1})b^\gamma)^{-m} 
    \leq  \frac{C}{\psi(b^{-1}) b^\gamma - 1},
 \end{align*}
 where the constant $C\in\RR_{++}$ is given by \eqref{phi_Holder}.
This implies that the family $\{ \vert Z_n\vert^p: n\in\NN\}$ is uniformly integrable,
 and that $Z_n$ converges to $Z$ as $n\to\infty$ $\PP$-almost surely (in particular, the random variable $Z$ is well-defined)
 yielding that $\vert Z_n\vert^p$ converges to $\vert Z\vert^p$ as $n\to\infty$ $\PP$-almost surely as well.
Consequently, the moment convergence theorem yields that $\EE(\vert Z\vert^p)<\infty$ and
 $\EE\big( \vert Z_n\vert^p - \vert Z\vert^p \big)\to 0$ as $n\to\infty$, as desired.
\end{proof}

\begin{Pro}\label{Pro7}
Let us consider the function $f$ defined by \eqref{help7} such that we always choose the sign $+1$
 or we choose alternating signs $+1$ and $-1$.
Suppose that $\psi(b^{-1}) > b^{-\gamma}$ and that $\psi$ is multiplicative.
Recall that $r=-\log_b(\psi(b^{-1}))\in(0,\gamma)$ is given in \eqref{beta_def}.
 \begin{itemize}
   \item[(i)] If $p=\frac{1}{r}$, then $\lim_{n\to\infty} V^{p,1}_n(f) = \EE(\vert Z\vert^p)$, where the random variable $Z$
               is defined in Lemma \ref{Lem6}.
   \item[(ii)] If $p\in[1,\frac{1}{r})$ and $\PP(Z\ne 0)>0$, then $\lim_{n\to\infty} V^{p,1}_n(f) = \infty$.
  \end{itemize}
\end{Pro}

\begin{proof}
Since $\psi$ is multiplicative, by \eqref{help15} with $t=1$, we get for all $n\in\NN$ and  $p\geq 1$,
 \begin{align}\label{help20}
   V^{p,1}_n(f) = \big( (\psi(b^{-1}))^p b \big)^n \EE\big(\vert Z_n \vert^p\big)
                =  \big( \psi(b^{-1}) \, b^{\frac{1}{p}} \big)^{np} \EE\big(\vert Z_n \vert^p\big).
 \end{align}
Note also that $\frac{1}{r}>\frac{1}{\gamma}\geq 1$.

(i): If $p=\frac{1}{r}$, then $b^{\frac{1}{p}} = b^r = (\psi(b^{-1}))^{-1}$, yielding that
 $\big( \psi(b^{-1}) \,b^{\frac{1}{p}} \big)^{np}=1$ for all $n\in\NN$.
Hence \eqref{help20} and Lemma \ref{Lem6} imply (i).

(ii): If $p\in [1,\frac{1}{r})$, then $b^{\frac{1}{p}} > b^r = (\psi(b^{-1}))^{-1}$, yielding that
 $\big( \psi(b^{-1})\, b^{\frac{1}{p}} \big)^{np} \to\infty$ as $n\to\infty$.
By Lemma \ref{Lem6}, we get $\EE\big(\vert Z_n \vert^p\big)\to \EE(\vert Z\vert^p)$ as $n\to\infty$,
 where $\EE(\vert Z\vert^p)\in\RR_{++}$ due to the assumption $\PP(Z\ne 0)>0$.
Hence \eqref{help20} implies (ii).
\end{proof}

\begin{Rem}\label{Rem_Z_choice}
The proofs of Lemma \ref{Lem6} and Proposition \ref{Pro7} (see Equation \eqref{help20})
 show why we restricted ourselves in the present section to the two cases
 $\xi_m=1$ for all $m\in\ZZ_+$, or $\xi_m=(-1)^m$ for all $m\in\ZZ_+$.
Namely, in the general case, we would have had to define $Z_n$ by $\sum_{m=1}^n \xi_{n-m}\big(\psi(b^{-1})b\big)^{-m} Y_m$,
 and we do not know whether it converges or not as $n\to\infty$ almost surely.
\proofend
\end{Rem}

\begin{Rem}
(i) If one always chooses the sign $+1$ in the definition \eqref{help7} of $f$, then
 part (i) of Proposition \ref{Pro7} is contained in Proposition A.2 in Schied and Zhang \cite{SchZha2}
 by choosing $\alpha:=\psi(b^{-1})$ and $t:=1$ (indeed, with the notations of Schied and Zhang \cite{SchZha2},
 $\EE_R(\vert Z\vert^p) = \EE(\vert Z\vert^p)$ due to the fact that $Z$ depends only on $R_m$, $m\in\NN$,
 but not on $W$).
If one always chooses the sign $+1$ in the definition \eqref{help7} of $f$,
 then part (ii) of Proposition \ref{Pro7} is also a consequence of Proposition A.2 in Schied and Zhang \cite{SchZha2}
 taking into account also that the existence of a nonzero continuous $\big(\frac{1}{r}\big)^{\mathrm{th}}$-variation
 function of $f$ implies that, for any $p\in[1,\frac{1}{r})$, we have
 $\lim_{n\to\infty} \sum_{k=0}^{b^n-1} \vert f((k+1)b^{-n}) - f(k b^{-n})\vert^p=\infty$,
 see part (ii) of Lemma \ref{Lem_change_point}.

(ii)
If one always chooses the sign $+1$ in the definition \eqref{help7} of $f$, then
 part (i) of Proposition \ref{Pro7} can be extended to $V^{p,t}_n(f)$ for $t\in[0,1]$ with
 the limit $t\cdot \EE\big(\vert Z \vert^p\big)$ as $n\to\infty$,
 see Schied and Zhang \cite[Proposition A.2]{SchZha2}.
If $p>\frac{1}{r}$, then, regardless whether $\PP(Z=0)=1$ holds or not, or $\psi$ is multiplicative or only submultiplicative,
 part (ii) of Theorem \ref{Thm5} yields that $\lim_{n\to\infty}V^{p,t}_n(f)=0$ for all $t\in[0,1]$.
Finally, we note that in Proposition \ref{Pro_Z} we will give an extension of part (i) of Proposition \ref{Pro7}
 in the special case $\PP(Z=0)=1$.
\proofend
\end{Rem}

The next corollary is a partial extension of the first statement of part (iii) of Theorem 2.1 in Schied and Zhang \cite{SchZha1}
  to the H\"older continuous case.

\begin{Cor}\label{Cor_bounded_var}
Let us consider the function $f$ defined by \eqref{help7} such that we always choose the sign $+1$
 or we choose alternating signs $+1$ and $-1$.
Suppose that $\psi(b^{-1}) > b^{-\gamma}$ and that $\psi$ is multiplicative.
If $f$ is of bounded variation, then $\PP(Z=0)=1$.
\end{Cor}

\noindent{\bf Proof.}
Suppose that $f$ is of bounded variation.
Then, taking into account that $f$ is continuous (see Lemma \ref{Lem_f_welldefined}),
 the limit $\lim_{n\to\infty}V^{1,1}_n(f)$ exists in $\RR_+$, and
 it equals the total variation of $f$ on $[0,1]$ (see, e.g., Natanson \cite[Theorem 2, Section 5, Chapter VIII]{Nat}).
Consequently, using part (ii) of Proposition \ref{Pro7} with $p=1$ (via contraposition rule),
 we have that $\PP(Z\ne 0)=0$, i.e., $\PP(Z=0)=1$.
\proofend

In the next Proposition \ref{Pro2}, in the case when we always choose the sign $+1$ in the definition \eqref{help7} of $f$,
 we present some sufficient conditions under which
 $\PP(Z\ne 0)>0$ holds (appearing as an assumption in part (ii) of Proposition \ref{Pro7}),
 where  $Z=\sum_{m=1}^\infty (\psi(b^{-1})b)^{-m} Y_m$ appears in Lemma \ref{Lem6}.
Proposition \ref{Pro2} is in fact the second part of Proposition A.2 in Schied and Zhang \cite{SchZha2},
 where the proof is left to the readers.
For completeness, we provide a detailed proof.

\begin{Pro}\label{Pro2}
Let us consider the function $f$ defined by \eqref{help7} such that we always choose the sign $+1$.
Suppose that $\psi(b^{-1}) > b^{-\gamma}$.
If $\{0\} \ne \big\{ \phi(b^{-k}) : k\in\NN \big\} \subseteq \RR_+$,
 then, for the random variable $Z=\sum_{m=1}^\infty (\psi(b^{-1})b)^{-m} Y_m$ appearing in Lemma \ref{Lem6}
 (by choosing always the sign $+1$), we have $\PP(Z\ne 0)>0$, which is equivalent to $\EE(\vert Z\vert)>0$.
\end{Pro}

\begin{proof}
Recall that, by \eqref{Y_estimate}, we have $\vert Y_m\vert\leq C b^{m(1-\gamma)}$, $m\in\NN$, where $C\in\RR_{++}$ is given by \eqref{phi_Holder}.
By the assumption, there exists $M\in\NN$ such that $\phi(b^{-M})>0$.
Choose $N\in\NN$ with $N>M$ and $\delta\in(0,\phi(b^{-M}))$ such that
 \begin{align}\label{help47}
  C \sum_{m=N}^\infty (\psi(b^{-1})b^\gamma)^{-m} < \phi(b^{-M})-\delta.
 \end{align}
For any $\delta\in(0,\phi(b^{-M}))$, such an $N$ exists, since $\sum_{m=0}^\infty (\psi(b^{-1})b^\gamma)^{-m}$ is convergent due to $\psi(b^{-1})b^\gamma>1$.

Recall that $(U_n)_{n\in\NN}$ is a sequence of independent and identically distributed random variables such that
 $U_1$ is uniformly distributed on the finite set $\{0,1,\ldots,b-1\}$.

If $\omega\in\{ U_1=0,U_2=0,\ldots,U_N=0\}$, then, by \eqref{help32} and \eqref{help33}, for all $m\in\{1,\ldots,N\}$, we get
 \begin{align}\label{help37}
   Y_m(\omega) = \lambda_{m,R_m(\omega)} = \lambda_{m,0}
               = b^m (\phi(b^{-m}) - \phi(0))
               = b^m \phi(b^{-m})
               \geq 0,
 \end{align}
 where in the last inequality we used the assumption that  $\big\{ \phi(b^{-k}) : k\in\NN \big\} \subseteq \RR_+$.

Hence if $\omega\in\{ U_1=0,U_2=0,\ldots,U_N=0\}$, then using also that $M\in\{1,\ldots,N-1\}$ and $\psi(b^{-1})\in(0,1)$, we have
 \begin{align*}
  \sum_{m=1}^{N-1} (\psi(b^{-1})b)^{-m} Y_m(\omega)
      \geq (\psi(b^{-1})b)^{-M} Y_M(\omega)
      \geq b^{-M} Y_M(\omega)
      = \phi(b^{-M})
      >0,
 \end{align*}
 where the last two steps follow by \eqref{help37} by choosing $m=M$ and the choice of $M$.
Therefore, if $\omega\in\{ U_1=0,\ldots,U_N=0\}$, by the reverse triangle inequality and \eqref{help47},  we have
 \begin{align*}
   \vert Z(\omega)\vert
        &\geq \left\vert \sum_{m=1}^{N-1} (\psi(b^{-1})b)^{-m} Y_m(\omega) \right\vert
               - \left\vert \sum_{m=N}^\infty (\psi(b^{-1})b)^{-m} Y_m(\omega) \right\vert\\
        &\geq \phi(b^{-M}) - \sum_{m=N}^\infty (\psi(b^{-1})b)^{-m} \vert Y_m(\omega)\vert\\
        &\geq \phi(b^{-M}) - C \sum_{m=N}^\infty (\psi(b^{-1})b)^{-m} b^{m(1-\gamma)}\\
        &= \phi(b^{-M}) - C \sum_{m=N}^\infty (\psi(b^{-1})b^\gamma)^{-m}
        >\delta,
 \end{align*}
 that is,
 \[
    \{ U_1=0,\ldots,U_N=0\} \subseteq \{ \vert Z\vert>\delta\}.
 \]
Since $U_1,\ldots,U_N$ are i.i.d.\ such that $U_1$ is uniformly distributed on the set $\{0,1,\ldots,b-1\}$, we get
 \[
   \PP(\vert Z\vert>\delta) \geq \PP(U_1=0,\ldots,U_N=0) = b^{-N} > 0,
 \]
 yielding that $\PP(Z \ne 0) \geq \PP(\vert Z\vert>\delta)>0$, as desired.
\end{proof}

\begin{Rem}\label{Rem_T_n_cond_sufficient}
Let us consider the function $f$ defined by \eqref{help7} such that we always choose the sign $+1$.
Suppose that $\psi$ is multiplicative, $\psi(b^{-1}) > b^{-\gamma}$, and that $\{0\} \ne \big\{ \phi(b^{-k}) : k\in\NN \big\} \subseteq \RR_+$.
Then $T_n=Z_n$, $n\in\NN$, where $T_n$ and $Z_n$ are given in \eqref{help_Tn}  and \eqref{help34}, 
 respectively (see the beginning of this section).
As we have seen in the proof of Lemma \ref{Lem6},  $Z_n$ converges to $Z$ as $n\to\infty$ $\PP$-almost surely, 
 where the random variable $Z$ is given in Lemma \ref{Lem6}.
Therefore, using also Proposition \ref{Pro2}, we have that 
 \begin{align*}
  \PP\left(\liminf_{n\to\infty} \vert T_n\vert > 0\right)
     = \PP(\vert Z\vert > 0)
     = \PP(Z\ne 0)>0.
 \end{align*}
Note that the condition $\PP(\liminf_{n\to\infty} \vert T_n\vert > 0)>0$ appeared in part (ii) of Theorem \ref{Thm_last_parts},
 and the previous argument shows that  it is satisfied under the conditions that $\xi_n=+1$, $n\in\ZZ_+$, 
  $\psi$ is multiplicative, $\psi(b^{-1}) > b^{-\gamma}$, and $\{0\} \ne \big\{ \phi(b^{-k}) : k\in\NN \big\} \subseteq \RR_+$.
\proofend
\end{Rem}

Next, we provide an improvement of part (i) of Proposition \ref{Pro7} by handling the case $p\in[\frac{1}{\gamma},\frac{1}{r}]$
 and $\PP(Z=0)=1$, where $Z$ is defined in Lemma \ref{Lem6} (which is also an improvement of Proposition A.2 in Schied and Zhang \cite{SchZha2}).

\begin{Pro}\label{Pro_Z}
Let us consider the function $f$ defined by \eqref{help7} such that we always choose the sign $+1$
 or we choose alternating signs $+1$ and $-1$.
Suppose that $\psi(b^{-1}) > b^{-\gamma}$, $\psi$ is multiplicative, and that $\PP(Z=0)=1$, where $Z$ is defined in Lemma \ref{Lem6}.
Recall that $r=-\log_b(\psi(b^{-1}))\in(0,\gamma)$ is given in \eqref{beta_def}.
 \begin{itemize}
   \item[(i)] If $p\in (\frac{1}{\gamma},\frac{1}{r}]$, then $\lim_{n\to\infty}V^{p,t}_n(f)=0$, $t\in[0,1]$,
   \item[(ii)] If $p=\frac{1}{\gamma}$, then $\limsup_{n\to\infty}V^{p,t}_n(f)\leq \left(\frac{C}{\psi(b^{-1}) b^\gamma-1}\right)^p t$, $t\in[0,1]$,
                where $C$ is given by \eqref{phi_Holder}.
 \end{itemize}
\end{Pro}

\noindent{\bf Proof.}
Since $\PP(Z=0)=1$, we have that $\PP$-almost surely for all $n\in\NN$,
 \[
   Z_n= - \sum_{m=n+1}^\infty \big(\pm \psi(b^{-1}) b\big)^{-m} Y_m,
 \]
 where $Z_n$, $n\in\NN$, are defined in \eqref{help34}.
Consequently, by \eqref{help15}, the non-negativity of $\psi$ and Minkowski's inequality
 for infinite sums of random variables (see, e.g., the proof of Theorem 3.4.1 in Cohn \cite{Coh}),
 for each $n\in\NN$ and $p\geq 1$ we get
 \begin{align*}
   (V^{p,t}_n(f))^{\frac{1}{p}}
    &= \Big( \big( (\psi(b^{-1}))^p b \big)^n \EE\big( \vert Z_n\vert^p  \bone_{\{ R_n\leq \lfloor tb^n\rfloor\}} \big) \Big)^{\frac{1}{p}} \\
    &= \big( (\psi(b^{-1}))^p b \big)^{\frac{n}{p}} \left( \EE \left(  \left\vert  \sum_{m=n+1}^\infty \big(\pm \psi(b^{-1}) b\big)^{-m} Y_m \right \vert^p
          \bone_{\{ R_n\leq \lfloor tb^n\rfloor\}}  \right)  \right)^{\frac{1}{p}}\\
    &\leq \big( (\psi(b^{-1}))^p b \big)^{\frac{n}{p}} \sum_{m=n+1}^\infty (\psi(b^{-1}) b)^{-m}  \big( \EE(\vert Y_m\vert^p)
       \bone_{\{ R_n\leq \lfloor tb^n\rfloor\}}  \big)^{\frac{1}{p}},\qquad t\in[0,1].
 \end{align*}
Hence, using \eqref{Y_estimate}, for each $n\in\NN$ and $t\in[0,1]$, we have
 \begin{align}\label{help38}
   \begin{split}
   (V^{p,t}_n(f))^{\frac{1}{p}}
     &\leq C\big( \big(\psi(b^{-1}))^p b \big)^{\frac{n}{p}} 
          (\PP(R_n\leq \lfloor tb^n\rfloor))^{\frac{1}{p}} \sum_{m=n+1}^\infty (\psi(b^{-1}) b)^{-m} b^{m(1-\gamma)}\\
     &= C\big( \big(\psi(b^{-1}))^p b \big)^{\frac{n}{p}}
          (\PP(R_n\leq \lfloor tb^n\rfloor))^{\frac{1}{p}} \sum_{m=n+1}^\infty (\psi(b^{-1}) b^\gamma)^{-m}\\
     &= C\big( \big(\psi(b^{-1}))^p b \big)^{\frac{n}{p}}
          (\PP(R_n\leq \lfloor tb^n\rfloor))^{\frac{1}{p}} \frac{(\psi(b^{-1}) b^\gamma)^{-(n+1)} }{1 - (\psi(b^{-1}) b^\gamma)^{-1}}\\
     &= \frac{C}{\psi(b^{-1}) b^\gamma-1} 
          (\PP(R_n\leq \lfloor tb^n\rfloor))^{\frac{1}{p}} \big( b^{1-\gamma p}\big)^{\frac{n}{p}}.
  \end{split}
 \end{align}
Since $R_n$ is uniformly distributed on $\{0,1,\ldots,b^n-1\}$, we have that
 $\PP(R_n\leq \lfloor tb^n\rfloor) = 1\wedge \frac{\lfloor tb^n\rfloor+1}{b^n}$, $n\in\NN$, $t\in[0,1]$,
 which yields that $\PP(R_n\leq \lfloor tb^n\rfloor) \to t$ as $n\to\infty$ for all $t\in[0,1]$.
If $p>\frac{1}{\gamma}$, then $\big( b^{1-\gamma p}\big)^{\frac{n}{p}} \to 0$ as $n\to\infty$,
 and if $p=\frac{1}{\gamma}$, then $\big( b^{1-\gamma p}\big)^{\frac{n}{p}} = 1$, $n\in\NN$.
Consequently, using \eqref{help38}, we obtain the statements of parts (i) and (ii).
\proofend

\begin{Rem}
Concerning part (i) of Proposition \ref{Pro_Z}, note that if $p=\frac{1}{r}$ and $\PP(Z=0)=1$, then part (i) of Proposition \ref{Pro7} also yields that
 $\lim_{n\to\infty}V^{p,1}_n(f)= \EE(\vert Z\vert^p) = 0$.
Further, if $p>\frac{1}{r}$, then, regardless whether $\PP(Z=0)=1$ holds or not, or $\psi$ is multiplicative or only submultiplicative,
 part (ii) of Theorem \ref{Thm5} yields that $\lim_{n\to\infty}V^{p,1}_n(f)=0$.
\proofend
\end{Rem}

The next corollary may be considered as a 'H\"older continuous' counterpart of the first statement of part (iii) of Theorem 2.1 in Schied and Zhang \cite{SchZha1}.

\begin{Cor}\label{Cor_Z}
Let us consider the function $f$ defined by \eqref{help7} such that we always choose the sign $+1$
 or we choose alternating signs $+1$ and $-1$.
Suppose that $\psi(b^{-1}) > b^{-\gamma}$ and that $\psi$ is multiplicative.
Then $\PP(Z=0)=1$ holds if and only if $\limsup_{n\to\infty}V^{1/\gamma,1}_n(f)<\infty$,
 where $Z$ is defined in Lemma \ref{Lem6}.
\end{Cor}

\noindent{\bf Proof.}
Let us suppose that $\PP(Z=0)=1$.
Then part (ii) of Proposition \ref{Pro_Z} implies that $\limsup_{n\to\infty}V^{1/\gamma,1}_n(f)<\infty$, as desired.

Suppose now that $\limsup_{n\to\infty}V^{1/\gamma,1}_n(f)<\infty$.
Since $\frac{1}{\gamma}\in[1,\frac{1}{r})$, where $r=-\log_b(\psi(b^{-1}))$ is given in \eqref{beta_def},
 part (ii) of Proposition \ref{Pro7} (via contraposition)
 yields that $\PP(Z\ne 0)=0$, i.e., $\PP(Z=0)=1$, as desired.
\proofend

\section{Riesz variation of Weierstrass-type functions along $b$-adic partitions}\label{Sec_Riesz_var}

First, we recall the notion of Riesz variation of a function, see, e.g., Appell et al.\ \cite[Definition 2.50]{AppBanMer}.

\begin{Def}\label{Def_Riesz_var}
Let $g:[0,1]\to \RR$ be a function.
For $p\geq 1$ and a partition $\cP_n:=\{0=t_0<t_1<\cdots< t_n=1\}$ of $[0,1]$, where $n\in\NN$, the nonnegative real number
 \[
  RV^{p}_n(g,\cP_n):=\sum_{k=0}^{n-1} \frac{\vert g(t_{k+1}) - g(t_k)\vert^p}{(t_{k+1}-t_k)^{p-1}}
 \]
 is called the $p^{\mathrm{th}}$-order Riesz variation of $g$ on $[0,1]$ with respect to the partition $\cP_n$.
Further, the (possibly infinite) number
 \[
    RV^{p}(g):=\sup\{ RV^{p}_n(g,\cP_n) : \text{$\cP_n$ is a partition of $[0,1]$, $n\in\NN$}\}
 \]
 is called the total $p^{\mathrm{th}}$-order Riesz variation of $g$ on $[0,1]$.
If $RV^{p}(g)<\infty$, then we say that $g$ has bounded $p^{\mathrm{th}}$-order Riesz variation on $[0,1]$.
\end{Def}

Note that $g:[0,1]\to\RR$ is of bounded variation if and only if $RV^{1}(g)<\infty$.
Further, if $g:[0,1]\to\RR$ has bounded $p^{\mathrm{th}}$-order Riesz variation on $[0,1]$ for some $p>1$,
 then $g$ is of bounded variation on $[0,1]$.
Indeed, by H\"older's inequality, for any $p>1$, for any $n\in\NN$ and any partition of $\cP_n=\{0=t_0<t_1<\cdots< t_n=1\}$ of $[0,1]$,
 we have that
 \begin{align*}
  \sum_{k=0}^{n-1} \vert g(t_{k+1}) - g(t_k) \vert
   &=  \sum_{k=0}^{n-1} \frac{\vert g(t_{k+1}) - g(t_k)\vert}{(t_{k+1} - t_k)^{1-\frac{1}{p}}}(t_{k+1} - t_k)^{1-\frac{1}{p}} \\
   &\leq \left(\sum_{k=0}^{n-1} \frac{\vert g(t_{k+1}) - g(t_k)\vert^p}{(t_{k+1} - t_k)^{p-1}} \right)^{\frac{1}{p}}
        \left(\sum_{k=0}^{n-1} (t_{k+1} - t_k)\right)^{1-\frac{1}{p}}\\
   &= (RV_n^p(g,\cP_n))^{\frac{1}{p}}\leq (RV^p(g))^{\frac{1}{p}},
 \end{align*}
 see also Appell et al.\ \cite[page 162]{AppBanMer}.
Moreover, it also holds that if $g$ has bounded $p^{\mathrm{th}}$-order Riesz variation on $[0,1]$ for some $p>1$,
 then $g$ is absolutely continuous (in particular, continuous and is of bounded variation (that was directly checked above as well),
 see Appell et al.\ \cite[Proposition 2.52]{AppBanMer}.
Finally, we mention that if $g:[0,1]\to\RR$ is Lipschitz continuous,
then it has bounded $p^{\mathrm{th}}$-order Riesz variation on $[0,1]$ for all $p\geq 1$
(see Appell et al.\ \cite[formula (2.94)]{AppBanMer}).
However, there exists a function $g:[0,1]\to\RR$, which is H\"older continuous with any exponent
$\mu\in(0,1)$, but $g$ does not have a bounded $p^{\mathrm{th}}$-order Riesz variation on $[0,1]$
 for any choice of $p\geq 1$ (see Appell et al.\ \cite[Example 2.53]{AppBanMer}).

Next, we investigate Riesz variation of a function $g:[0,1]\to\RR$ with respect to the $b$-adic partition of $[0,1]$, where $b\in\NN\setminus \{1\}$.
If $\Pi_n:=\{kb^{-n}: k=0,1,\ldots,b^n\}$, $n\in\NN$, is the $b$-adic partition of $[0,1]$,
 then for all $p\geq 1$, we have
 \begin{align}\label{Riesz_pth_var_Def}
  \begin{split}
  RV^{p}_n(g,\Pi_n) & = \sum_{k=0}^{b^n-1} \frac{\vert g((k+1)b^{-n}) - g(kb^{-n})\vert^p}{((k+1)b^{-n} - kb^{-n})^{p-1}}
                      = b^{n(p-1)}\sum_{k=0}^{b^n-1} \vert g((k+1)b^{-n}) - g(kb^{-n})\vert^p \\
                    & = b^{n(p-1)} V^{p,1}_n(g),
  \end{split}
 \end{align}
 where $V^{p,1}_n(g)$ is defined in \eqref{pth_var_Def}.

Parts (i), (ii) and (iii) of our next Theorem \ref{Thm_Riesz} are counterparts of Theorems \ref{Thm3}, \ref{Thm4} and \ref{Thm5}, respectively,
 in the sense that these results are about the asymptotic behaviour of $p^\mathrm{th}$-order Riesz variations
 (instead of $p^\mathrm{th}$-variations) of a Weierstrass-type function $f$ on $[0,1]$ (defined by \eqref{help7}) along the sequence of $b$-adic partitions.

\begin{Thm}\label{Thm_Riesz}
Let us consider a function $f$ defined by \eqref{help7},
 let $\Pi_n:=\{kb^{-n}: k=0,1,\ldots,b^n\}$, $n\in\NN$, where $b\in\NN\setminus \{1\}$, and let $p\geq 1$.
 \begin{itemize}
  \item[(i)] If $\psi(b^{-1}) < b^{-\gamma}$, then
              \[
                 \limsup_{n\to\infty} \frac{RV^{p}_n(f,\Pi_n)}{b^{p(1-\gamma)n}} \leq \left(\frac{LC}{1-\psi(b^{-1})b^\gamma}\right)^p,
              \]
              where $C\in\RR_{++}$ is given by \eqref{phi_Holder}.
  \item[(ii)] If $\psi(b^{-1}) = b^{-\gamma}$, then
              \[
                 \limsup_{n\to\infty} \frac{RV^{p}_n(f,\Pi_n)}{n^p b^{p(1-\gamma)n}} 
                 \leq L^p \left(  1+ \frac{1}{\log_b(2)} \right)^p \left(  C + 2\sup_{x\in\RR} \vert \phi(x)\vert \frac{1}{1 - \psi(b^{-1})} \right)^p.
              \]
  \item[(iii)] If $\psi(b^{-1}) > b^{-\gamma}$, then
              \[
                 \limsup_{n\to\infty} \frac{RV^{p}_n(f,\Pi_n)}{b^{p(1-r)n}} \leq \left(\frac{LC}{\psi(b^{-1})b^\gamma - 1}\right)^p,
              \]
              where we recall that $r=-\log_b(\psi(b^{-1}))\in(0,\gamma)$ is given in \eqref{beta_def}.
\end{itemize}
\end{Thm}

Concerning the normalization factors for $RV^{p}_n(f,\Pi_n)$ in Theorem \ref{Thm_Riesz},
 note that $b^{p(1-\gamma)n} < n^p b^{p(1-\gamma)n}$ for all $n\geq 2$,
 and $n^p b^{p(1-\gamma)n} < b^{p(1-r)n}$ for large enough $n\in\NN$, since $r\in(0,\gamma)$.

\noindent{\bf Proof of Theorem \ref{Thm_Riesz}.}
(i): Suppose that $\psi(b^{-1}) < b^{-\gamma}$.
Using \eqref{Riesz_pth_var_Def} and the equality \eqref{help35}
 (note that for the derivation of \eqref{help35}, we only used \eqref{help14}), we have that
 \begin{align*}
   RV^{p}_n(f,\Pi_n)   =  b^{n(p-1)} b^{n(1-\gamma p)} \EE(\vert W_n\vert^p) 
                                   =  b^{p(1-\gamma)n} \EE(\vert W_n\vert^p), \qquad n\in\NN,
 \end{align*}
 which together with part (i) of Theorem \ref{Thm3} imply the assertion of part (i).

(ii): Suppose that $\psi(b^{-1}) = b^{-\gamma}$.
Using \eqref{Riesz_pth_var_Def} and the inequality in \eqref{help40}
 (note that for the derivation of the inequality in \eqref{help40}, we only used that $\psi(b^{-1})=b^{-\gamma}$),
 there exists a constant $C_1\in\RR_{++}$ such that for all $n\in\NN$ we have
 \begin{align*}
  RV^{p}_n(f,\Pi_n) \leq  b^{n(p-1)} \cdot C_1^p n^p b^{n(1-\gamma p)}  = C_1^p n^p b^{p(1-\gamma)n},
 \end{align*}
 where the constant $C_1$ can be chosen as in \eqref{constant_C1} with $\eta:=1$.
This implies the assertion of part (ii).

(iii): Suppose that $\psi(b^{-1}) > b^{-\gamma}$.
Using \eqref{Riesz_pth_var_Def} and \eqref{help36}
 (note that for the derivation of \eqref{help36}, we only used \eqref{help14}), for all $n\in\NN$, we have that
 \begin{align*}
   RV^{p}_n(f,\Pi_n)
   = b^{n(p-1)} b^{n(1-rp)} \EE( \vert T_n\vert^p ) = b^{p(1-r)n} \EE( \vert T_n\vert^p).
 \end{align*}
This together with part (i) of Theorem \ref{Thm5} imply the assertion of part (iii).
\proofend

\medskip

Next, we consider a special case of Theorem \ref{Thm_Riesz}.
Namely, suppose that $\gamma=1$ (Lipschitz continuity), $\psi(b^{-1})=b^{-1}$, $\psi$ is multiplicative and we choose $\xi_m=1$ for all $m\in\ZZ_+$,
 in the definition \eqref{help7} of $f$.
Then, using \eqref{help15} with $t=1$ and \eqref{Riesz_pth_var_Def}, for all $n\in\NN$ and $p\geq 1$, we get that
 \begin{align}\label{help43}
  RV^{p}_n(f,\Pi_n) = \EE\left( \left\vert  \sum_{m=1}^n  Y_m \right\vert^p \right),
 \end{align}
 where $Y_m$, $m\in\NN$, are given in \eqref{help33}, and part (ii) of Theorem \ref{Thm_Riesz} yields that
 \begin{align}\label{help42}
   \limsup_{n\to\infty} \frac{RV^{p}_n(f,\Pi_n)}{n^p} < \infty.
 \end{align}
If, in addition, $\phi(t)=\min_{z\in\ZZ}\vert t-z\vert$, $t\in\RR$ (which corresponds to Takagi functions) and $b$ is even,
 then Schied and Zhang \cite[Proposition 3.3]{SchZha1} showed that $Y_m$, $m\in\NN$, are independent and identically distributed such that
 $\PP(Y_1=1)=\PP(Y_1=-1)=\frac{1}{2}$, and hence in this case $\sum_{m=1}^n  Y_m$, $n\in\NN$, is nothing else but a usual symmetric random walk.
Using part (ii) of Lemma 2 in Basrak and Kevei \cite{BasKev} (combinations of Jensen-, Marcinkiewicz-Zygmund- and Rosenthal inequalities),
 for all $p\geq 1$, there exists a constant $K_p\in\RR_{++}$ such that for all $n\in\NN$, we have
 \[
  \EE\left( \left\vert \sum_{m=1}^n  Y_m \right\vert^p \right)
    \leq K_p n^{\max(1,\frac{p}{2})} \EE(\vert Y_1\vert^p)
    = K_p n^{\max(1,\frac{p}{2})}.
 \]
If $p>1$, then $\frac{1}{n^p}n^{\max(1,\frac{p}{2})} \to0 $ as $n\to\infty$, and hence, by \eqref{help43},
 we get
 \[
   \lim_{n\to\infty} \frac{RV^{p}_n(f,\Pi_n)}{n^p} = 0,
 \]
 which improves \eqref{help42} in the case $\phi(t)=\min_{z\in\ZZ}\vert t-z\vert$, $t\in\RR$ and $b$ is even.

\appendix

\vspace*{5mm}

\noindent{\bf\Large Appendix}

\vspace*{3mm}

We recall a result on the decomposition of submultiplicative functions
  in terms of the product of a power function and another appropriate function
 due to Finol and Maligranda \cite[Theorem 1]{FinMal}, and we also provide some non-trivial examples of submultiplicative functions.
 
Let $I\subseteq \RR_{++}$ be a subset of such that $xy\in I$ whenever $x,y\in I$.
A measurable function $g:I\to \RR_{++}$ is called submultiplicative if $g(xy)\leq g(x)g(y)$, $x,y\in I$.

Finol and Maligranda \cite[Theorem 1]{FinMal} proved that if $g:(0,1)\to\RR_{++}$ is a submultiplicative function, then the limit
 \[
     \lim_{x\downarrow 0} \frac{\ln(g(x))}{\ln(x)} =: \alpha\in \RR \qquad \text{exists,}
 \]
 and
 \[
    g(x) = x^\alpha h(x), \qquad x\in(0,1),
 \]
 where $h:(0,1)\to \RR$ is a function satisfying $h(x)\geq 1$ for all $x\in(0,1)$, and
 $\lim_{x\downarrow 0} x^\vare h(x)=0$ for all $\vare>0$.
Further, if $\lim_{x\downarrow 0} g(x)=0$, then $\alpha\in\RR_{++}$.

We also give some examples of non-trivial submultiplicative functions on $\RR_{++}$.
For any $A\in[1,\infty)$, the functions $\psi_i:\RR_{++}\to\RR_{++}$, $i\in\{1,2,3,4\}$, given by
 \begin{align*}
   &\psi_1(x):=A+\vert \ln(x) \vert, \qquad x\in\RR_{++},\\
   &\psi_2(x):=x^A(1+\vert \ln(x)\vert), \qquad x\in\RR_{++},\\
   &\psi_3(x):=A+\vert \sin(\ln(x)) \vert, \qquad x\in\RR_{++},\\
   &\psi_4(x):= x^A(1+\vert \sin(\ln(x)) \vert), \qquad x\in\RR_{++},
 \end{align*}
 are submultiplicative, see Maligranda \cite[Examples 3, 4 and 5]{Mal}.
Note that $\lim_{x\downarrow 0} \psi_1(x)=\infty$, the limit $\lim_{x\downarrow 0} \psi_3(x)$ does not exist, and
 $\lim_{x\downarrow 0} \psi_i(x) =0$ for $i\in\{2,4\}$.


%

\end{document}